\documentclass[a4paper]{amsart}
\usepackage{amssymb,amscd,verbatim,latexsym}
\usepackage{enumerate}
\usepackage{tikz-cd}
\usepackage{tikz}
\usepackage[all,pdf]{xy}

\usepackage[T1]{fontenc}
\usepackage[american]{babel}
\usepackage[utf8]{inputenc}
\usepackage{csquotes}

\usepackage{lmodern}
\usepackage{microtype}

\usepackage[backend=biber,giveninits=true,sortcites=true,url=false,isbn=false]{biblatex}
\usepackage{mathscinet}

\usepackage[ocgcolorlinks]{hyperref}
\hypersetup{
  colorlinks = true,
  allcolors = [rgb]{0.15, 0.37, 0.41},
  linktocpage = true}

\usepackage[mathscr]{eucal}

\usepackage{color}
\newcommand\ede{ \, := \, }

\newcommand{\supp}{\operatorname{supp}}

\newcommand\pullback{\sp{\downarrow\downarrow}}
\newcommand{\tto}{\rightrightarrows}
\newcommand{\st}{\rightrightarrows}
\newcommand{\dr}{\rightrightarrows}
\newcommand{\eqdef}{:=}
\newcommand\mathbfPsi{\mathbf \Psi}

\newcommand{\Prim}{\operatorname{Prim}}

\newcommand{\pbg}{\downdownarrows}

\newcommand{\CC}{\mathbb C}

\newcommand{\N}{\mathbb N}

\newcommand{\RR}{\mathbb R}
\newcommand{\R}{\mathbb R}

\newcommand{\ZZ}{\mathbb Z}

\newcommand\pa{{\partial}}
\newcommand\Cstar{C\sp{\ast}}

\newcommand\Cs[1]{C\sp{\ast}(#1)}
\newcommand\rCs[1]{C_r\sp{\ast}(#1)}

\newcommand{\A}{\mathcal A}

\newcommand{\maC}{\mathcal C}
\newcommand{\cC}{\mathcal C}

\newcommand{\maF}{\mathcal F}
\newcommand{\maG}{\mathcal G}
\newcommand{\cG}{\mathcal G}
\newcommand{\G}{\mathcal G}
\newcommand{\maH}{\mathcal H}
\newcommand{\cH}{\mathcal H}

\newcommand{\cJ}{\mathcal J}
\newcommand{\maK}{\mathcal K}
\newcommand{\maL}{\mathcal L}

\newcommand{\maP}{\mathcal P}
\newcommand{\maR}{\mathcal R}

\newcommand{\cV}{\mathcal V}
\newcommand{\V}{\mathcal V}
\newcommand{\maW}{\mathcal W}

\newcommand{\g}{\mathfrak{g}}
\newcommand{\B}{\mathcal{B}}

\renewcommand{\S}{\mathbb{S}}

\newcommand{\de}{{\rm d}}
\newcommand{\maJ}{\mathcal J}
\def\pa{\partial}
\def\d{\partial}
\newtheorem{theorem}{Theorem}[section]
\newtheorem{thm}{Theorem}[section]
\newtheorem{proposition}[theorem]{Proposition}
\newtheorem{corollary}[theorem]{Corollary}

\newtheorem{lm}[theorem]{Lemma}

\theoremstyle{definition}
\newtheorem{definition}[theorem]{Definition}
\newtheorem{defn}[theorem]{Definition}
\theoremstyle{remark}
\newtheorem{rem}[theorem]{Remark}
\newtheorem{remark}[theorem]{Remark}
\newtheorem{example}[theorem]{Example}
\newtheorem{ex}[theorem]{Example}

\author[C. Carvalho]{Catarina Carvalho} \address{Dep. Matem\'{a}tica,
    Instituto Superior T\'{e}cnico, University of Lisbon, Av. Rovisco
    Pais, 1049-001 Lisbon, Portugal }
\email{catarina.carvalho@math.tecnico.ulisboa.pt}

\author[R. C\^ome]{R\'emi C\^ome} \address{Institut \'Elie Cartan de Lorraine, Universit\'e de Lorraine, 3 rue Augustin Fresnel, 57000 Metz, France}
\email{remi.come@univ-lorraine.fr}

\author[Y. Qiao]{Yu Qiao} \address{School of Mathematics and
  Information Science, Shaanxi Normal University, Xi'an, 710119,
  China} \email{yqiao@snnu.edu.cn}

\thanks{Carvalho was partially supported by Funda\c c\~ao para a Ci\^{e}ncia e a Tecnologia, Portugal,
UID/MAT/04721/2013. Qiao
  was partially supported by NSF of China 11301317 and the Foundational
  Research Funds for the Central Universities GK201803003.
}

\date\today

\title[Boundary Action Groupoids and Layer Potentials]{Gluing action groupoids: Fredholm conditions and layer potentials}

\begin{document}
\maketitle

\begin{abstract}
  We introduce a new class of groupoids, called \emph{boundary action groupoids}, which are obtained by gluing reductions of action groupoids. We show that such groupoids model the analysis on many singular spaces, and we give several examples. Under some conditions on the action of the groupoid, we obtain Fredholm criteria for the pseudodifferential operators generated by boundary action groupoids. Moreover, we show that layer potential groupoids for conical domains constructed in
 an earlier paper (Carvalho-Qiao, Central European J. Math., 2013) are
 both Fredholm groupoids and boundary action groupoids, which enables us to deal with many analysis problems
 on singular spaces in a unified way. As an application,
 we obtain Fredholm criteria for operators on layer potential groupoids.
\end{abstract}

\tableofcontents

\section{Introduction}

\label{sec:introduction}

The aim of this paper is to study some algebras of psuedodifferential operators on open manifolds that are ``regular enough at infinity'' (in that they can be compactified to a \enquote{nice} manifold with corners). More specifically, we look for a full characterization of the \emph{Fredholm} operators, meaning those that are invertible modulo a compact operator.

\subsection{Background}
\label{sub:background_and_main_results}

Let $M_0$ be a smooth manifold and $P :H^s(M_0) \to H^{s-m}(M_0)$ a differential operator of order $m$, acting between Sobolev spaces. When $M_0$ is a closed manifold, it is well-known that $P$ is Fredholm if, and only if, it is \emph{elliptic} \cite{Hor1985}. This is a important result, that has many applications to partial differential equations, spectral theory and index theory. A great deal of work has been done to obtain such conditions when $M_0$ is not compact, because in that case being elliptic is no longer a sufficient condition to be Fredholm.

A possible approach is to consider manifolds $M_0$ that embed as the interior of a compact manifold with corners $M$ and differential operators that are ``regular'' near $\d M$. This is the approach followed by Melrose, Monthubert, Nistor, Schulze and many others: see for instance \cite{Mel1993,Sch1991,LN,DLR2015}. A differential operator $P$ in this setting is Fredholm if, and only if, it is elliptic \emph{and} a family of \emph{limit operators} $(P_x)_{x \in \d M}$ is invertible; we shall give more details below.

Lie groupoids have been proven to be an effective tool to obtain Fredholmness results and to model analysis on singular spaces in general (see for instance \cite{MonthubertSchrohe, AJ2, DebordLescure1, DebordLescure2, LN, Monthubert01, MonthubertNistor, MRen} and the references therein for a small sample of applications). One general advantage of this strategy is that, by associating a Lie groupoid to a given singular problem, not only are we able to use groupoid techniques, but we also get automatically a groupoid $C^*$-algebra and well-behaved pseudodifferential calculi naturally affiliated to
this $C^*$-algebra \cite{ASkandalis2, LMN, LN, Monthubert03, NWX,vanErpYuncken}. In many situations, the family of limit operators can be obtained from suitable representations of the groupoid $C^{*}$-algebra, so Fredholmness may be studied through representation theory.

Recently \cite{CNQ, CNQ17}, the concept of \emph{Fredholm groupoids} was introduced as, in some sense, the largest class of Lie groupoids for which such Fredholm criteria hold with respect to a natural class of representations, the {regular representations} (see Section \ref{s.fredholm} for the precise definitions). A characterization of such groupoids is given relying on the notions of \emph{strictly spectral} and \emph{exhaustive} families of representations, as in \cite{nistorPrudhon, Roch}.
The associated non-compact manifolds are named \emph{manifolds with amenable ends}, since certain isotropy groups at infinity are assumed to be amenable. This is the case for manifolds with cylindrical and poly-cylindrical ends, for
manifolds that are asymptotically Euclidean, and for manifolds that
are asymptotically hyperbolic, and also manifolds obtained by iteratively blowing-up singularities. In \cite{CNQ}, we discuss these examples extensively, and show how the Fredholm groupoid approach provides a unified treatment for many singular problems.

Many interesting Fredholm groupoids are action groupoids: this approach has been followed by Georgescu, Iftimovici and their collaborators \cite{GI2006,GN2017b,MNP2017,Man2017}. These authors considered the smooth action of a Lie group $G$ on a compact manifold with corners $M$. In this setting, the limit operators $(P_x)_{x \in \d M}$ are obtained as ``translates at infinity'' of $P$ under the action of $G$. This point of view allows the study of operators with singular coefficients, such as those occuring in the $N$-body problem.

The first part of our work studies a special class of groupoids, named \emph{boundary action groupoids}, which is obtained by gluing action groupoids (in a sense made precise in the paper). Boundary action groupoids have a simple local structure, and occur naturally in many of the examples discussed above.

In the second part of the paper, we consider conical domains. Our original motivation comes here from the study of boundary problems for elliptic differential equations,
namely by applications of the classical {method of layer potentials}, which reduces differential equations to \emph{boundary} integral equations \cite{Fol, McL, Tay2}. One typically wants to invert an operator of the form \enquote{$\frac{1}{2} +K$} on suitable
function spaces on the boundary of some domain $\Omega$. If the boundary is $\cC^1$, then the
integral operator $K$ is compact on $L^2(\pa\Omega)$ \cite{FJR}, so the operator $\frac{1}{2}+K$ is Fredholm and
we can apply the classical Fredholm theory to solve the Dirichlet problem. But if
there are singularities on the boundary, as in the case of conical domains, this result is not necessarily true \cite{Els, FJL, LP, Verchota}.
Suitable groupoid $C^{*}$-algebras and their representation theory are then a means to provide the right replacement for the compact operators, and the theory of Fredholm groupoids is suited to yield the desired Fredholm criteria.

In \cite{CQ13},  we associated to $\Omega$, or
more precisely to $\pa \Omega$, a \emph{layer potentials groupoid}  over the (desingularized) boundary that aimed to provide the right setting to study invertibility and Fredholm problems as above. As a space, we have
\begin{equation*}
  \cG:= \left(\bigsqcup\limits_{i} (\pa\omega_i \times \pa\omega_i) \times \RR^+ \right)  \bigsqcup \Omega_0 \times \Omega_0 \quad  \tto \quad M:=\left(\bigcup\limits_{i } \partial\omega_{i} \times [0,1) \right)
    \bigcup \Omega_0,
\end{equation*}
where $\bigsqcup$ is the disjoint union, $\Omega_0$ is the smooth part of $\pa \Omega$, and the local cones have bases  $\omega_{i} \subset
S^{n-1}$, where $\omega_i$'s are domains with smooth boundary, $i=1,...,l$. The space of units $M$ can be thought of as a \enquote{desingularized boundary}. The limit operators in this case, that is, the operators over $M\setminus \Omega_{0}$, have dilation invariant kernels  on $(\pa\omega_i\times \pa\omega_i) \times \RR^+ $, which eventually yield a family of Mellin convolution operators on $( \pa\omega_i) \times \RR^+ $, one for each local cone. This fact was one of the original motivations in our definition.
In \cite{CQ13}, we were able to obtain Fredholm criteria making use of the machinery of pseudodifferential operators on Lie manifolds. These Fredholm criteria are formulated on weighted Sobolev spaces,
 we refer the reader to \cite{Kon,MR} and references therein for some details.
 Let us state their definition here: let $r_{\Omega}$ be the  smoothed  distance
function to the set of conical points of $\Omega$ and let $\Omega_0$ be the smooth part of $\partial\Omega$. Recall that the $m$-th
Sobolev space on $\pa\Omega$ with weight $ r_{\Omega}$ and index $a$ is defined by
\begin{equation*}\label{eq.def.ws}
  \maK_{a}^m(\pa\Omega)=\{u\in L^2_{\text{loc}}(\pa\Omega), \, \,
  r_{\Omega}^{|\alpha|-a}\partial^\alpha u\in L^2(\Omega_0), \,\,\,\text{for
    all}\,\,\, |\alpha|\leq m\}.
\end{equation*}
We have the following isomorphism \cite{BMNZ}:
\begin{equation*}
   \maK^{m}_{\frac{n-1}{2}}(\partial\Omega)\simeq
  H^{m}(\partial'\Sigma(\Omega),g), \quad  \mbox{  for all  } m\in \mathbb{R},
\end{equation*}
where $\Sigma(\Omega)$ is a desingularization of $\Omega$ and $g=r_\Omega^{-2}g_e$, with $g_e$ the standard Euclidean metric, and $\partial'\Sigma(\Omega)$ is the union of the hyperfaces that are not at infinity in $\pa \Sigma(\Omega)$, which can be identified with a desingularization of $\pa \Omega$ (see Section \ref{s.LP_groupoids}).

\subsection{Overview of the main results}
\label{sub:overview_of_the_main_results}

The purpose of the the present paper is two-fold. The first part of the paper introduces the class of \emph{boundary action groupoids}, that are obtained by \emph{gluing} a family of action groupoids $(G_i \rtimes M_i)_{i \in I}$ (in a sense made precise below). We will show that this setting recovers many interesting algebras of pseudodifferential operators. Moreover, we rely on the results in \cite{CNQ2017} to obtain the following Fredholm condition:

\begin{theorem} \label{thm:intro_fredholm}
  Let $\G\dr M$ be a boundary action groupoid with unique dense orbit $U \subset M$, and let $P \in L^m_s(\G) \supset \Psi^m(\G)$. Assume that the action of $\G$ on $\d M$ is trivial and that, for each $x \in \d M$, the isotropy group $\G_x^x$ is amenable. Then $P: H^s(U) \to H^{s-m}(U)$ is Fredholm if, and only if
  \begin{enumerate}
    \item $P$ is elliptic, and
    \item $P_x : H^s(\G^x_x) \to H^{s-m}(\G_x^x)$ is invertible for all $x \in \d M$.
  \end{enumerate}
\end{theorem}
Here $L^m_s(\G)$ is the completion of the space $\Psi^m(\G)$ of order-m pseudodifferential operators on $\G$ in the topology of $\mathcal{L}(H^s,H^{s-m})$, which act on $U$ in a natural way (see Subsection \ref{ss.ops.grpds}). The notion of ellipticity for $P$ is the usual one: its principal symbol $\sigma(P) \in \Gamma(T^*U)$ should be invertible outside the zero-section. Theorem \ref{thm:intro_fredholm} extends directly to pseudodifferential operators acting between sections of vector bundles.

We will show that the assumptions of Theorem \ref{thm:intro_fredholm} are satisfied in many natural situations, for instance when one wishes to study geometric operators on asymptotically Euclidean or asymptotically hyperbolic manifolds. The limit operators $P_x$ are right-invariant differential operators on the groups $\G_x^x$ and are of the same type as $P$. For example, if $P$ is the Laplacian on $M_0$, then $P_x$ is also the Laplacian for a right-invariant metric on the group $\G_x^x$.

For the second part of the paper, we relate both the Fredholm groupoid and boundary groupoid approaches to the study of \emph{layer potential operators} on domains with \emph{conical singularities}.
We consider here bounded domains with conical points $\Omega$ in
$\mathbb{R}^n$, $n \ge 2$, that is, $\overline{\Omega}$ is locally diffeomorphic
to a cone with smooth, \emph{possibly disconnected}, base. (If $n=2$, we allow $\Omega$ to be a domain with cracks. See Section  \ref{s.LP_groupoids} for the precise definitions.) We consider the layer potentials groupoid defined in \cite{CQ13}, which is a groupoid over over a desingularization of the boundary,  
and we place it in the setting of boundary action groupoids. Moreover, we show independently that the layer potentials groupoid associated to the boundary of a conical domain is indeed a Fredholm groupoid (Theorem \ref{thm.LPFred}).
 We obtain  Fredholm criteria naturally
and are able to extend them to  a space of operators  that contains $L^{2}$-inverses.

Applying the results of \cite{CNQ,CNQ2017} for Fredholm groupoids, we obtain our main result (Theorems \ref{thm.fredholm}). As above, the space $L^{m}_s(\maG)$ is the completion of the space of order-$m$ pseudodifferential operators $\Psi^m(\maG)$ with respect to the operator norm on Sobolev spaces (see Section \ref{ss.ops.grpds}).

\begin{theorem}
Let $\Omega \subset \mathbb{R}^n$ be a conical domain without cracks with the set of conical points
$\Omega^{(0)}=\{p_1,p_2,\cdots, p_l\}$, with possibly disconnected cone bases $\omega_{i}\subset S^{n-1}$.

\smallskip

Let $\maG\tto M=\pa'\Sigma(\Omega)$
be the layer potential groupoid as in Definition \ref{gpd1}. Let $P\in L^{m}_s(\maG)\supset \Psi^m(\maG)$ and $s \in \RR $.
Then $P : \maK_{\frac{n-1}{2}}^s(\pa\Omega) \to \maK_{\frac{n-1}{2}}^{s-m}(\pa\Omega)$ is Fredholm if, and only if,
  \begin{enumerate}
    \item $P$ is elliptic, and
    \item the Mellin convolution operators
     $$P_{i}: =\pi_{p_{i}}(P) : H^s(\RR^+\times \pa\omega_i;g) \to H^{s-m}(\RR^+\times \pa\omega_i;g)\,,$$
	     are invertible for all $i=1,..., l$,
	        where the metric $g := r^{-2}_\Omega \, g_e$ with $g_e$ the Euclidean metric.
  \end{enumerate}
\end{theorem}

The above theorem also holds, with suitable modifications, for polygonal domains with ramified cracks.

The layer potentials groupoid constructed here is related to the so-called  $b$-groupoid (Example \ref{bgrpd}) associated to the manifold
with smooth boundary $\pa' \Sigma(\Omega)$. The $b$-groupoid can be used to recover Melrose's $b$-calculus
\cite{MelroseAPS}. If the boundaries of the local cones bases are connected, then the two groupoids coincide (note that it is often the case that the boundaries are disconnected, for instance take $n=2$). In general, our pseudodifferential calculus
contains the compactly supported $b$-pseudodifferential operators, in that our
groupoid contains the $b$-groupoid as an open
subgroupoid. The main difference at the groupoid level is that in the usual $b$-calculus there is no interaction between the different faces at each conical point.

In \cite{QL18}, Li and the third-named author applied the techniques of pseudodifferential operators on Lie groupoids to the method of layer potentials on plane polygons (without cracks) to obtain the invertibility of operators $I \pm K$ on suitable weighted Sobolev spaces on the boundary, where $K$ is the double layer potential operators (also called Neumann-Poincar\'{e} operators) associated to the Laplacian and the polygon. The Lie groupoids used in that paper are exactly the groupoids we constructed in \cite{CQ13}, which will be shown to be Fredholm in this paper. Moreover, the third-named author used a similar idea to make a connection between the double layer potential operators on three-dimensional wedges and (action) Lie groupoids in \cite{Qiao18}.

However, for domains with cracks, the resulting layer potential operators are no longer Fredholm. These issues will be addressed in a forthcoming paper.

\subsection{Contents of the paper}
\label{sub:contents_of_the_paper}

We start in Sections \ref{sub:lie_groupoids_and_lie_algebroids}  to \ref{sub:groupoid_examples}
by reviewing some relevant facts about  Lie groupoids, including Lie algebroids,  groupoid $C^{*}$-algebras and pseudodifferential operators on Lie groupoids . We discuss the necessary facts about Lie groupoids and their algebroids, and give several important examples.

In Section \ref{ss.fredholm}, we review the definition of Fredholm groupoids and their characterization, relying on exhaustive families of representations, resulting on Fredholm criteria for operators on Fredholm groupoids.

Section \ref{sec:gluing_groupoids} introduces one of the main constructions of the paper, which is the gluing of a family of locally compact groupoids $(\G_i)_{i \in I}$. We give two different conditions that are sufficient to define a groupoid structure on the gluing $\G = \bigcup_{i\in I} \G_i$, and show some properties of the gluing. When each $\G_i$ is a Lie groupoid, we describe the Lie algebroid of the glued groupoid $\G$.

We define the class of \emph{boundary action groupoids} in Subsection \ref{sub:boundary_action_groupoids}. We give some examples of boundary action groupoids which occur naturally when dealing with analysis on open manifolds. We then explain the construction of the algebra of differential operators generated by a Lie groupoid $\G$, and prove the Fredholm condition given by Theorem \ref{thm:intro_fredholm}.

In the remaining sections, we consider the case of  layer potential groupoids on conical domains.
In Section \ref{s.LP_groupoids}, we describe the construction of  the relevant groupoids and give their main properties in the case with no cracks (Section \ref{ss.groupoidnocrack}).

In Section \ref{s.FredCond}, we start with checking that the layer potential groupoid is always a boundary action groupoid. We show moreover that such groupoids are Fredholm and obtain Fredholm criteria for operators on layer potential groupoids.

\vspace{0.2cm}
\emph{Acknowledgements:} We would like to thank Victor Nistor for useful discussions and suggestions, as well as the editors for fostering this joint work.

\vspace{0.3cm}
\section{Fredholm Groupoids}\label{s.fredholm}

The aim of this section is to recall some basic definitions and constructions regarding groupoids, and especially Lie groupoids (we refer to Renault's book \cite{renault80} for locally compact groupoids, and Mackenzie's books \cite{Mackenzie87, Mackenzie05} for Lie groupoids). In  Subsection  \ref{ss.fredholm} we recall the definitions and results concerning Fredholm Lie groupoids as in \cite{CNQ, CNQ17}.

\vspace{0.2cm}

\subsection{Lie groupoids and Lie algebroids}
\label{sub:lie_groupoids_and_lie_algebroids}

Let us begin with the definition of a groupoid, as in \cite{Mac1987,Ren1980}.
\begin{defn}
  A \emph{groupoid} is a small category in which every morphism is invertible.
\end{defn}
\begin{rem} \label{rem:defn_groupoid_better}
  It is often more useful to see a groupoid $\G$ as a set of \emph{objects} $\G^{(0)}$ and a set of \emph{morphisms} $\G^{(1)}$. We will often identify $\G$ with its set of morphisms $\G^{(1)}$. Any element $g \in \G$ has a domain $d(g)$ and range $r(g)$ in $\G^{(0)}$, as well as an inverse $\iota(g) \eqdef g^{-1} \in \G$. To every object $x \in \G^{(0)}$ corresponds a (unique) unit map $u(x) \in \G$. Finally, the product of two morphisms defines a map $\mu$ from the set of \emph{composable arrows}
  \begin{equation*}
    \G^{(2)} \eqdef \{(g,h) \in \G \times \G, d(g) = r(h) \}
  \end{equation*}
   to $\G$. The groupoid $\G$ is completely determined by the pair $(\G^{(0)}, \G^{(1)})$, together with the five structural maps $d,r,\iota,u$ and $\mu$ \cite{Ren1980,Mac1987}.
\end{rem}

We now fix some notations for later. When $(g,h) \in \G^{(2)}$, the product will be written simply as $gh \eqdef \mu(g,h)$.
We shall also write $\G \st M$ for a groupoid $\G$ with objects $\G^{(0)} = M$. Finally, let $A \subset M$, and put $\G_A \eqdef d^{-1}(A)$ and $\G^A \eqdef r^{-1}(A)$. 
The groupoid $\G|_A \eqdef \G^A \cap \G_A$ will be called the \emph{reduction} of $\G$ to $A$.
The \emph{saturation} of $A$ is the subset of $M$ defined by
$$\G \cdot A =\{r(g) \mid g\in \G, \, d(g) \in A \} = r\big(d^{-1} (A)  \big).$$
In particular, if $A$ is a point $\{x\} \subset M$,
then $\G \cdot x$ is the {\em orbit} of $x$ in $M$.

\begin{defn}
  A \emph{locally compact groupoid} is a groupoid $\G \st M$ such that :
  \begin{enumerate}
    \item $\G$ and $M$ are locally compact spaces, with $M$ Hausdorff,
    \item the structural morphisms $d,r,\iota,u$ and $\mu$ are continuous, and
    \item $d : \G \to M$ is surjective and open.
  \end{enumerate}
\end{defn}
Note that these conditions imply that $r : \G \to M$ is surjective and open as well. Only the unit space $M$ is required to be Hausdorff in the general definition, so $\maG$ may be non-Hausdorff. In this paper, we will \emph{not} assume the space $\G$ to be Hausdorff, and we will always specify when it is so. We give several examples of groupoids in Subsection \ref{sub:groupoid_examples} below.

\vspace{0.2cm}

\emph{Lie groupoids} are groupoids with a smooth structure. In the analysis of problems on singular spaces,
it is crucial to distinguish between smooth manifolds without corners and manifolds with boundary or corners. The manifolds we consider here may have corners, which occurs in many applications; for example, this is the case when one has to remove a singularity by blowing-up a submanifold \cite{Nis2015, Vas2001}. Thus, in our setting, a \emph{manifold} $M$ is a \emph{second-countable} space that is locally modelled on open subsets of $[0,\infty[^n$, with smooth coordinate changes \cite{Nis2015}. Note that $M$ is not necessarily Hausdorff, unless stated explicitly. By a \emph{smooth} manifold, we shall mean a manifold without corners.

By definition, every point $p \in M$ of a {\em manifold with corners} has a coordinate neighborhood
diffeomorphic to $[0,1)^k \times (-1,1)^{n-k}$ such that the transition functions are smooth.
The number $k$ is called the {\em depth} of the point $p$.
The set of {\em inward pointing tangent vectors} $v\in T_p(M)$ defines
a closed cone denoted by $T^+_p(M)$.

A smooth map $f; M_1 \rightarrow M_2$ between
two manifolds with corners is called a {\em tame submersion} provided that
$df(v)$ is an inward pointing vector of $M_2$ if and only if $v$ is an inward pointing vector of $M_1$.
Lie groupoids are defined as follows.

\begin{defn}
  A \emph{Lie groupoid} is a groupoid $\G \st M$ such that
  \begin{enumerate}
    \item $\G$ and $M$ are manifolds with corners, with $M$ Hausdorff,
    \item the structural morphisms $d,r, \iota$ and $u$ are smooth,
    \item the range $d$ is a tame submersion, and
    \item the product $\mu$ is smooth.
  \end{enumerate}
\end{defn}
We remark that $(3)$ implies that each fiber $\maG_x=d^{-1}(x) \subset \maG$ is a smooth manifold (without corners) \cite{CNQ, Nistor_Comm}.
In particular, the Lie groupoids we use are locally compact and second countable spaces, but they are not necessarily Hausdorff (and many important examples, coming in particular from foliation theory \cite{Con1982}, yield non Hausdorff groupoids). A slightly more general class of groupoids, also useful in applications, is that of \emph{continuous family groupoids}, for which we assume smoothness along the fibers only, and continuity along the units \cite{Pat2001,LMN2000}.

If $G$ is a Lie group (which is a particular example of Lie groupoid, see Example \ref{ex:lie_groups_bundle}), then its tangent space over the identity element has a structure of Lie algebra, induced by the correspondence with right-invariant vector fields on $G$. The corresponding construction for a general Lie groupoid is that of a \emph{Lie algebroid} \cite{Mac1987,Mac2005}.

\begin{defn}
  Let $M$ be a manifold with corners and $A \to M$ a smooth vector bundle. We say that $A$ is a \emph{Lie algebroid} if there is a Lie algebra structure on the space of sections $\Gamma(A)$, together with a vector bundle morphism $\rho : A \to TM$ covering the identity, and such that the induced morphism
    \begin{equation*}
      \rho : \Gamma(A) \to \Gamma(TM)
    \end{equation*}
    is a Lie algebra morphism. In that case, the map $\rho$ is called the \emph{anchor} of $A$.
\end{defn}

\begin{ex} \label{ex:lie_algebroid_of_a_groupoid}
  Let $\G \st M$ be a Lie groupoid. Then $u : M \to \G$ is an embedding and we can consider
  \begin{equation*}
    \A\G \eqdef (\ker d_*)|_M = \bigsqcup_{x\in M} T_x\G_x,
  \end{equation*}
  which is a vector bundle over $M$. The smooth sections of $\A\G$ are in one-to-one correspondence with the right-invariant vector fields on $\ker d_*$, which form a Lie algebra. This gives $\A\G$ a Lie algebroid structure, whose anchor is given by $r_*$.
\end{ex}

The definition of Lie algebroid morphism was given for instance in \cite{Mac1987, Mac2005}.

\begin{defn}
  Let $A \to M$ and $B \to N$ be two Lie algebroids. A \emph{morphism of Lie algebroids} from $A$ to $B$ is pair $(\Phi,\phi)$ such that
  \begin{enumerate}
    \item $\phi : M \to N$ is a smooth map and $\Phi : A \to B$ a vector bundle morphism covering $\phi$,
    \item $\Phi$ induces a Lie algebra morphism $\Gamma(A) \to \Gamma(B)$, and
    \item the following diagram is commutative:
      \begin{equation*}
        \begin{tikzcd}
          A \arrow[r, "\Phi"] \arrow[d, "\rho_A"] & B \arrow[d, "\rho_B"] \\
        TM \arrow[r, "\phi_*"] & TN,
        \end{tikzcd}
      \end{equation*}
      with $\rho_A$ and $\rho_B$ the respective anchor maps.
  \end{enumerate}
\end{defn}

A Lie algebroid $A \to M$ is said to be \emph{integrable} whenever there is a Lie groupoid $\G \st M$ such that $\A\G$ is isomorphic to $A$. Not every Lie algebroid is integrable: the relevant obstruction is discussed in \cite{CF2011}. However, some classical results of Lie algebra theory remain true in this more general case. In order to state those, we shall say that a Lie groupoid $\G \to M$ is $d$-connected (respectively $d$-simply-connected) if each of its $d$-fibers $\G_x$ is connected (respectively simply-connected), for every $x \in M$. A proof of the following two results may be found in \cite{Nis2000,MM2002}.

\begin{thm}[Lie I]
  \label{thm:Lie1}
  Let $A \to M$ be a Lie algebroid. If $A$ is integrable, then there is a (unique) $d$-simply-connected groupoid integrating $A$.
\end{thm}

\begin{thm}[Lie II] \label{thm:Lie2}
  Let $\phi : A \to B$ be a morphism of integrable Lie algebroids, and let $\maG$ and $\maH$ be integrations of $A$ and $B$. If $\G$ is $d$-simply-connected, then there is a (unique) morphism of Lie groupoids $\Phi : \maG \to \maH$ such that $\phi=\Phi_*$.
\end{thm}

\vspace{0.2cm}

\subsection{Groupoid $C^{*}$-algebras}\label{ss.groupoidCalg}

We assume all our locally compact groupoids $\G \tto M$ to be endowed with a fixed (right) Haar system, denoted $(\lambda_x)_{x \in M}$. Here $\lambda_x$ is a measure on the $d$-fiber $\G_x$, and the family $(\lambda_x)_{x\in M}$ should be invariant under the right action of $\G$ and satisfy a continuity condition \cite{Ren1980}. All Lie groupoids have well-defined (right) Haar systems. For simplicity, we also assume that all our groupoids are \emph{Hausdorff} (some adjustments to the definition must be considered in the non-Hausdorff case \cite{Con1982}).

 \smallskip

 To any locally compact groupoid $\maG$ (endowed with a Haar system),
there are associated two basic $C^*$-algebras, the {\em full} and {\em
  reduced} $C^*$-algebras $\Cs{\maG}$ and $\rCs{\maG}$, whose
  definitions we recall now. Let $\maC_c(\maG)$ be the space of continuous, complex valued, compactly supported functions on $\maG$. We endow $C_c(\G)$ with the convolution product
  \begin{equation*}
    \varphi * \psi(g) = \int_{\G_{d(g)}} \varphi(gh^{-1})\psi(h) \de\lambda_{d(g)}(h)
  \end{equation*}
and the usual involution $\varphi^*(g) \ede \overline{\varphi(g^{-1})}$, thus defining an associative $*$-algebra structure \cite{Ren1980}.

  There exists a natural algebra norm on $\maC_c(\maG)$ defined by
\begin{equation*}
  \| \varphi\|_1 \ede \max \, \Bigl\{ \, \sup_{x\in M}\int_{\maG_x} \vert
  \varphi\vert\de\lambda_x, \, \sup_{x\in M}\int_{\maG_x}\vert
  \varphi^*\vert\de\lambda_x \, \Bigr\}.
\end{equation*}
The completion of $\maC_c(\maG)$ with respect to the norm $\|\cdot
\|_1$ is denoted $L^1(\maG)$.

For any $x\in M$, the algebra $\maC_c(\maG)$ acts as bounded operators on $L^2(\maG_x,\lambda_x)$.
Define for any $x\in M$ the
\emph{regular} representation $\pi_x\, \colon\, \maC_c(\maG) \to
\maL(L^2(\maG_x,\lambda_x))$ by
\begin{equation*}
  (\pi_x(\varphi)\psi) (g) \ede \varphi * \psi(g) \ede \int_{\maG_{d(g)}}
  \varphi(gh^{-1}) \psi(h) d\lambda_{d(g)}(h) \,, \quad \varphi \in
  \maC_c(\maG) \,.
\end{equation*}
We have $\|\pi_x(\varphi)\|_{L^2(\maG_x)} \leq\|\varphi \|_{1}$.

\begin{definition}\label{def.regular}
We  define the \emph{reduced $C\sp{\ast}$-algebra}
$C\sp{\ast}_{r}(\maG)$ as the completion of $\maC_c(\maG)$ with
respect to the norm
\begin{equation*}
  \| \varphi\|_r \ede \sup\limits_{x \in M}\|\pi_x(\varphi)\| \,
\end{equation*}
The \emph{full $C\sp{\ast}$-algebra} associated to $\maG$, denoted
$C\sp{\ast}(\maG)$, is defined as the completion of $\maC_c(\maG)$
with respect to the norm
\begin{equation*}
  \| \varphi\| \ede \sup\limits_\pi\|\pi(\varphi)\| \,,
\end{equation*}
where $\pi$ ranges over all {\em contractive} $*$-representations of
$\maC_c(\maG)$, that is, such that $\| \pi(\varphi)\|\leq \| \varphi \|_1$, for all $\varphi \in \maC_c(\maG)$.

The groupoid~$\maG$ is said to be \emph{metrically amenable} if the
canonical surjective $*$-homomorphism $\Cstar(\maG) \to
\Cstar_{r}(\maG)$, induced by the definitions above, is also
injective.

\end{definition}

Let $\maG \tto M$ be a second countable, locally compact groupoid with
a Haar system. Let $U \subset M$ be an open $\maG$-invariant subset,
$F := M \smallsetminus U$.
Then, by the classic results of \cite{MRW87, MRW96, renault91},  $C\sp{\ast}(\maG_U)$ is a closed
  two-sided ideal of $C\sp{\ast}(\maG)$ that yields the short exact
  sequence
\begin{equation}\label{renault.exact_item1}
  0\to C\sp{\ast}(\maG_U) \to
  C\sp{\ast}(\maG)\mathop{\longrightarrow}\limits^{\rho_F}
  C\sp{\ast}(\maG_{ F })\to 0 \,,
\end{equation}
where $\rho_{F}$ is the (extended) restriction map.
 If $\maG_F$ is metrically
  a\-me\-nable, then one also has the exact sequence
\begin{equation}\label{renault.exact_item3}
  0\to C\sp{\ast}_{r}(\maG_U)\to C\sp{\ast}_{r}(\maG)
  \mathop{\xrightarrow{\hspace*{1cm}}}\limits^{(\rho_F)_{r}}
  C\sp{\ast}_{r}(\maG_F)\to 0 \,.
\end{equation}

It follows from the Five Lemma that if the groupoids $\maG_F$ and $\maG_U$ (respectively, $\maG$)
are metrically amenable, then $\maG$ (respectively, $\maG_U$) is
also metrically amenable.
We notice that these exact sequences correspond to a disjoint union decomposition $\maG = \maG_F \sqcup
\maG_U.$

\vspace{0.2cm}
\subsection{Pseudodifferential operators on Lie groupoids}\label{ss.ops.grpds}
We recall in this subsection the construction of pseudodifferential operators on
Lie groupoids \cite{LMN, LN, Monthubert01, Monthubert03, MonthubertPierrot, NWX}.
Let $P=(P_x)_{x\in M}$ be a smooth family of
pseudodifferential operators acting on $\maG_x:=d^{-1}(x)$. The family
$P$ is called \emph{right-invariant} if $P_{r(g)}U_g=U_gP_{d(g)}$, for all
$g\in \maG$, where
\begin{equation*}
  U_g : \maC^\infty(\maG_{d(g)}) \rightarrow
  \maC^\infty(\maG_{r(g)}), \,\ (U_gf)(g')=f(g'g).
\end{equation*}
Let $k_x$ be the distributional kernel of $P_x$, $x\in M$. The \emph{support} of $P$ is
\begin{equation*}
  \text{supp}(P):= \overline{\bigcup_{x\in M}\text{supp}(k_x)}  \subset \{(g,g'), \ d(g)=d(g')\} \subset \maG \times \maG,
\end{equation*}
  since $\text{supp}(k_x)\subset
  \maG_x\times\maG_x$. Let $\mu_1(g',g) :=
  g'g^{-1}$. The family $P = (P_x)$ is called \emph{uniformly
    supported} if its \emph{reduced support} $\text{supp}_\mu(P) :=
  \mu_1(\text{supp}(P))$ is a compact subset of $\maG$.

\begin{definition}\label{def.C*}
 The space $\Psi^{m}(\maG)$ of \emph{pseudodifferential operators of
   order $m$ on a Lie groupoid} $\maG$ with units $M$ consists
   of smooth families of pseudodifferential operators $P=(P_x)_{x\in M}$, with $P_x\in \Psi^m(\maG_x)$, which are {uniformly
   supported} and {right-invariant}.
\end{definition}

We also denote $\Psi^\infty(\maG) := \bigcup_{m\in
  \mathbb{R}}\Psi^m(\maG)$ and $\Psi^{-\infty}(\maG) :=
\bigcap_{m\in \mathbb{R}}\Psi^m(\maG)$. If $k_x$ denotes the distributional kernel of $P_x$, $x\in M$,
then the formula
$$k_P(g):=k_{d(g)}(g, d(g))$$
defines a distribution on the groupoid $\maG$,
with $\text{supp}( k_p) = \text{supp}_\mu(P)$ compact, smooth
outside $M$ and given by an oscillatory integral on a neighborhood of $M$.
If $P\in \Psi^{-\infty}(\maG)$, then $P$ is a convolution
operator with
smooth, compactly supported kernel. Thus $\Psi^{-\infty}(\maG)$
identifies with the smooth convolution algebra $\maC_c^\infty(\maG)$.  In
particular, we can define
\begin{equation*}
  \|P\|_{L^1(\maG)} := \sup\limits_{x\in M} \Big\{ \
  \int_{\maG_x}|k_P(g^{-1})|\, d\mu_x(g),\,\, \int_{\maG_x}|k_P (g)|\,
  d\mu_x(g)\ \Big\}.
\end{equation*}

The algebra $\Psi^\infty(\G)$ is \enquote{too small} to contain resolvents of differential operators. Thus we consider suitable closures. For each $x\in M$, the {\em regular representation} $\pi_x$ extends to $\Psi^\infty(\maG)$,
by defining by $\pi_x(P)=P_x$. If $P \in \Psi^{-n-1}(\maG)$, then $P$ has a continuous distribution kernel and
$$\|\pi_x(P)\|_{L^2(\maG_x)} \leq\|P\|_{L^1(\maG)}.$$
If $P \in \Psi^0(\G)$, then $\pi_x(P) \in \B(L^2(\G_x))$. We define the {\em reduced $C^*$--norm} by
\begin{equation*}
  \|P\|_r = \sup\limits_{x\in M}\|\pi_x(P)\| = \sup\limits_{x\in
  M}\|P_x\|.
\end{equation*}
Let $L^0_0(\G)$ be the completion of $\Psi^0(\G)$ for the reduced norm. Note that $C^*_r(\G)$ is the completion of $\Psi^{-\infty}(\G)$ for $\|.\|_r$, hence $C^*_r(\G)$ embeds as an ideal of $L^0_0(\G)$.

We consider similar completions for operators of arbitrary order. To that end, we endow the fibers $(\G_x)_{x \in M}$ with a right-invariant metric and consider the associated Laplacian $\Delta \in \Psi^2(\G)$. The Sobolev space $H^s(\G_x)$ is defined as the domain of $(1+\Delta_x)^{ \frac{s}{2} }$ if $s \ge 0$ (and by duality for $s < 0$). We set $L^m_s(\G)$ to be the completion of $\Psi^m(\G)$ with respect to the norm
\begin{equation*}
  \|P\|_{m,s} \ede \sup_{x \in M} \|P_x\|,
\end{equation*}
where $P_x$ is seen as a bounded operator from $H^s(\G_x)$ to $H^{s-m}(\G_x)$ \cite{Grosse.Schneider.2013, LN}.
 Moreover, let
\begin{equation*}
 \maW^{m}(\maG) \ede \Psi^{m}(\maG) + \bigcap_{s\in\R} L^{-\infty}_{s}(\maG)\,.
\end{equation*}
Then $\maW^{m}(\maG) \subset L^{m}_{s}(\maG)$ and $\maW^{\infty}(\maG)$
is an algebra of pseudodifferential operators that contains the inverses
of its $L^2$-invertible operators.

Assume now that there is an open, dense and $\G$-invariant subset $U \subset M$ such that $\G|_U \simeq U \times U$; this will be a natural requirement in Subsection \ref{ss.fredholm}. In that case each fibers $\G_x$, for $x \in U$ is diffeomorphic to $U$. Therefore any right-invariant metric on the fibers $(\G_x)_{x \in M}$ induces a metric on $U$. The regular representations $\pi_x$ are all equivalent when $x \in U$, so we define the \emph{vector representation}
\begin{equation*}
  \pi_0 : C^*_r(\G) \to \B(L^2(U))
\end{equation*}
as the equivalence class of all $\pi_x$, where $x \in U$. Then $\pi_0$ extends as a $C^*$-algebra morphism
\begin{equation*}
  \pi_0 : L^0_0(\G) \to \B(L^2(U)),
\end{equation*}
and as a bounded linear map
\begin{equation*}
  \pi_0 : L^m_s(\G) \to \B(H^s(U),H^{s-m}(U)).
\end{equation*}

\begin{remark} \label{rem:Hausdorff_pi_0_injective}
  When $\G$ is Hausdorff, which will be the case below, a result of Khoshkam and Skandalis \cite{KS2002} implies that the vector representation $\pi_0 : C_r(\G)\to \B(L^2(U))$ is \emph{injective}. In that case, the algebra $L^0_0$ embeds as a subalgebra of $\B(L^2(U))$.
\end{remark}

\vspace{0.2cm}

\subsection{Examples of Lie groupoids}
\label{sub:groupoid_examples}

Let us now give a few common examples of Lie groupoids that will have a role in our constructions.

\begin{example}[Bundles of Lie groups] \label{ex:lie_groups_bundle}
Any Lie group $G$ can be regarded as a Lie groupoid $\maG=G$ with exactly one unit $M= \{e\}$,
the identity element of $G$, and obvious structure maps. Its Lie algebroid is the Lie algebra of the group. In that case $\Psi^m(\maG) \simeq \Psi^m_{\text{prop}}(G)^G$,
the algebra of right translation invariant and properly supported pseudodifferential operators on $G$.

More generally, we can let $\maG \tto B$ be a locally trivial bundle of groups, with fiber a Lie group $G$. In that case $d = r$, and $\G$ is metrically amenable if, and only if, the group $G$ is amenable.
\end{example}

The following examples are more involved, and will be useful in what follows.

\begin{ex}[The pair groupoid] \label{ex:pair_groupoid}
  Let $M$ be a smooth manifold, and consider the Lie groupoid $\G = M\times M$ as the
groupoid having exactly one arrow between any two units, with structural morphisms as follow: the domain is $d(x,y) = y$, the range $r(x,y) = x$, and the product is given by $(x,y)(y,z) = (x,z)$. Thus $u(x) = (x,x)$ and $\iota(x,y) = (y,x)$. This example is called the \emph{pair groupoid} of $M$. The Lie algebroid of $\G$ is isomorphic to $TM$.

In this case, we have $\Psi^m(\maG) \simeq \Psi^m_{\text{comp}}(M)$,
the algebra of compactly supported pseudodifferential operators on $M$. For any $x\in M$, the regular
representation $\pi_x$ defines an isomorphism between
$C\sp{\ast}(M\times M)$ and the ideal of compact operators in
$\maL(L^2(M))$. In particular, all pair groupoids are metrically
amenable.
\end{ex}

\begin{ex}[Actions groupoids] \label{ex:action_groupoid}
  Let $X$ be a smooth manifold and $G$ a Lie group acting on $X$ smoothly and from the right. The \emph{action groupoid} generated by this action is the graph of the action, denoted by $X \rtimes G$. Its set of arrows is $X\times G$, together with the structural morphisms $r(x,g) \eqdef x$, $d(x,g) \eqdef x \cdot g^{-1}$ and $(x,h)(x \cdot h^{-1},g) \eqdef (x,gh)$.

  The Lie algebroid of $X \rtimes G$ is denoted by $X \rtimes \g$. As a vector bundle, it is simply $X \times \g$, where $\g$ is the Lie algebra of $G$. Its Lie bracket is generated by the one of $\g$: namely, if $\tilde{\xi}, \tilde{\eta}$ are constant sections of $X \times \g$ such that $\tilde{\xi}(x) = \xi$ and $\tilde{\eta}(x) = \eta$ for all $x \in \xi$, then $[\tilde{\xi},\tilde{\eta}]_{X\rtimes \g}$ is the constant section on $\xi$ everywhere equal to $[\xi,\eta]_\g$. The anchor $\rho : X \rtimes \g \to TX$ is given by the \emph{fundamental vector fields} generated by the action:
  \begin{equation*}
    \rho(x,\xi) = \left. \frac{d}{dt} \right|_{t=0} \left( x \cdot \exp(t\xi) \right)
  \end{equation*}
  for all $x \in X$ and $\xi \in \g$. The study of such groupoids relates to that of crossed-product algebras, which have been much studied in the literature \cite{Wil2007} (see also \cite{GI2006,MNP2017}).

One case of interest here is when $\maG:=[0,\infty) \rtimes (0,\infty) $ is the transformation groupoid
with the action of $(0,\infty)$ on $[0,\infty)$ by dilation. Then the $C^*$-algebra associated to $\cG$ is the algebra of Wiener-Hopf
operators on $\RR^+$, and its unitalization is the algebra of Toeplitz
operators  \cite{MRen}.
\end{ex}

\begin{ex}[Fibered pull-back groupoids]\ \label{ex:pullback_groupoid}
  Let $M,N$ be manifolds with corners, and $f : M \to N$ a surjective tame submersion. Assume that we have a Lie groupoid $\maH \st N$.  An important
generalization of the pair groupoid is the \emph{fibered pull-back} of $\maH$ along $f$, defined by
  \begin{equation*}
    f^{\pbg}(\maH) = \left\{ (x, g , y) \in M \times \G \times M, r(g) = f(x), d(g) = f(y) \right\}
  \end{equation*}
  with units $M$. The domain and range are given by $d(x,g,y) = y$ and $r(x,g,y) = x$. The product is $(x,g,y)(y,g',y') \eqdef (x,gg',y')$.

  The groupoid $f^\pbg(\maH)$ is a Lie groupoid,   which is a subgroupoid of the product of the pair groupoid $X\times X$ and $\maH$, and whose Lie algebroid is given by the \emph{thick pull-back}
  \begin{equation*}
    f^\pbg(\A\maH) \eqdef \left\{ (\xi,X) \in \A\maH \times TM, \rho(\xi) = f_*(X) \right\}.
  \end{equation*}
  See \cite{Mac1987,Mac2005,CNQ2017} for more details.

Let $\maH \tto B$ be a locally trivial bundle of groups (so $d = r$)
 with fiber a locally compact group $G$.  Also, let $f : M \to
 B$ be a continuous map that is a local fibration. Then $f\pullback
 (\maH)$ is a locally compact groupoid with a Haar system. If $G$ is a Lie group, $M$ is a manifold with corners and $f$ is a tame submersion, then $f\pullback
 (\maH)$ is a Lie groupoid.
Again,  it is metrically amenable if, and only if,
 the group $G$ is
 amenable.
\end{ex}

\begin{example}[Disjoint unions]\label{ex.help-for-lp}
Let  $M$ be a {smooth } manifold and let $\maP=\{M_{i}\}_{i=1}^{p}$ be a \emph{finite} partition of $M$ into smooth disjoint, closed submanifolds $M_{i}\subset M$ (since $\maP$ is finite, $M_{i}$ is also open, $i=1,..., p$, and the sets $M_{i}$ are always given by unions of connected components of $M$).
Let $f: M \to \maP$, $x \mapsto M_{i}$, with $x\in M_{i}$, be the quotient map.
Then $\maP$ is discrete and $f$ is locally constant,
so any Lie groupoid $\maH \tto \maP$ yields a Lie groupoid $f\pullback(\maH)\tto M$.
In particular, if $\maH=\maP$ as a (smooth, discrete) manifold, then
$f\pullback(\maP)$ is the topological disjoint union
$$f\pullback(\maP)= \bigsqcup_{i=1}^{p}(M_{i} \times M_{i}).$$
Let $G$ be a Lie group and $\maH := B \times
G$, the product of a manifold and a Lie group, then
$$f\pullback(\maH)= \bigsqcup_{i}^{p}(M_{i} \times M_{i}) \times G.$$
\end{example}

\begin{example}[$b$-groupoid]\label{bgrpd}

Let $M$ be a manifold with smooth boundary and let $\cV_b$ denote the
class of vector fields on $M$ that are tangent to the boundary.
The associated groupoid was
defined in \cite{MelroseAPS, Monthubert03, NWX}.
Let
$$
    \cG_b:=\Big( \bigcup\limits_{j} \RR^+\times (\pa_j M)^2\Big)
    \quad \cup \quad M_0^2,
$$
where $M_0^2$ denotes the pair groupoid of $M_0:= int(M)$ and $\pa_j
M$ denote the connected components of $\pa M$. Then $\cG_b$ can be
given the structure of a Lie groupoid with units $M$, given locally by a transformation groupoid.
It integrates the so-called $b$-tangent bundle ${^bTM}$, that is, $A(\cG_b)= {^bTM}$, the Lie algebroid whose space of sections is given by vector fields tangent to the boundary. The
pseudodifferential calculus obtained is Melrose's small $b$-calculus with compact supports. See
\cite{MelroseAPS, Monthubert03, MonthubertPierrot, NWX} for details.
\end{example}

\smallskip
\subsection{Fredholm groupoids}\label{ss.fredholm}

The classes of examples we have seen in the previous section, as wide ranging as they are, all share one common property: they fall in the framework of Fredholm groupoids (under certain suitable conditions for each case).
Fredholm groupoids were introduced in \cite{CNQ, CNQ17} as groupoids for which an operator is Fredholm if,
and only if, its principal symbol and all its \enquote{boundary restrictions}
are invertible (in a sense to be made precise below). We
review their definition and properties in this subsection.
\vspace{0.2cm}

Let $\maG \tto M$ be a Lie groupoid with $M$ compact, and assume that $U \subset M$
is an open, $\maG$-invariant subset such that $\maG_U \simeq U \times U$ (the pair groupoid, see Example \ref{ex:pair_groupoid}).
Recall from Subsection \ref{ss.ops.grpds} the definition of the vector representation $\pi_0 \colon C_r\sp{\ast}(\maG) \to \maL(L^2(U))$. We recall the following definition from \cite{CNQ2017}:

\begin{definition}
A Lie groupoid $\maG \tto M$ is called a {\em Fredholm Lie groupoid} if
\begin{enumerate}
\item there exists an open, dense, $\maG$-invariant subset $U\subset M$ such that $\maG_U \simeq U \times U$;
\item for any $a\in C^*_r(\maG)$, we have that $1+\pi_0(a)$ is Fredholm if,
and only if, all $1+\pi_x(a)$, $x\in F:=M\backslash U$ are invertible.
\end{enumerate}
\end{definition}

As an open dense $\G$-orbit, the set $U$ is uniquely determined by $\G$. Moreover, a simple observation is that $F:=M \backslash U$ is closed and $\maG$-invariant.
We shall keep this notation throughout the paper. Note also that two regular representations $\pi_x$ and
$\pi_y$ are unitarily equivalent if, and only if, there is $g \in \maG$ such that
$d(g) = x$ and $r(g) = y$, that is, if $x, y$ are in the same orbit of
$\maG$ acting on $M$. In particular, one only needs to verify (2) for a representative of each orbit of $\maG_{F}$.

In \cite{CNQ, CNQ17}, we gave easier-to-check conditions for a groupoid $\maG$ to be Freholm, depending on properties of representations of $C_r\sp{\ast}(\maG)$. We review briefly the main notions, see \cite{nistorPrudhon, Roch} for details.

Let $A$ be a $C^*$-algebra. Recall that a two-sided ideal
$I \subset A$ is said to be {\em primitive} if it is the kernel of an irreducible
representation of $A$. We denote by $\Prim(A)$ the set of primitive ideals of $A$
and we equip it with the hull-kernel topology {(see \cite{DavidsonBook, williamsBook} for more details)}.
Let $\phi$ be a representation of $A$. The {\em support} $\supp(\phi)\subset \Prim(A)$
is defined to be the set of primitive
ideals of $A$ that contain $\ker(\phi)$.

The following definition appeared in \cite{nistorPrudhon} :
\begin{definition}
A set of $\maF$ of representations of a
$C^*$-algebra $A$ is said to be {\em exhaustive} if $\Prim(A)= \bigcup_{\phi\in \maF} \supp(\phi)$,
that is, if any irreducible representation is weakly contained in some $\phi\in\maF$.
\end{definition}

If $A$ is unital, then a set $\maF$ of representations of $A$ is called {\em
  strictly spectral} if  it characterizes invertibility in $A$, in that $a \in A$ is invertible if, and only if, $\phi(a)$ is invertible
for all $\phi \in \maF$. If $A$ does not have a unit, we replace
$A$ with $A\sp{+} := A \oplus \CC$ and $\maF$ with $\maF\sp{+} := \maF
\cup \{\chi_0 : A\sp{+} \to \CC\}$,
where $\maF$ is regarded as a family of representations of $A^+$ and $\chi_0$ is the representation defined by $\chi_0|_A = 0$ and $\chi_0(1) =1$.
Note that strictly spectral families of representations consist
of non-degenerate representations, and any non-degenerate representation of
a (closed, two-sided) ideal in a $C\sp{\ast}$-algebra always has a unique
extension to the whole algebra \cite{nistorPrudhon}.

It was proved in \cite{nistorPrudhon, Roch} that, if $\maF$ is exhaustive, then $\maF$ is strictly
spectral, and the converse also holds  if $A$ is separable.
 That is, in the separable case, exhaustive families suffice to characterize invertibility in $A$. In this paper, we shall work mainly with the notion of exhaustive families and assume this equivalence throughout.

The next result was given  in \cite{CNQ, CNQ17} and  gives a characterization of Fredholm
groupoids. For a groupoid $\maG$, we usually denote by
 $\maR(\maG)$ the set of its regular representations.

\begin{theorem}  \label{thm.Fredholm.Cond}
Let $\maG \tto M$ be a Hausdorff Lie groupoid and $U$ an open, dense, $\maG$-invariant subset  such that $\maG_U \simeq U \times U$, $F=M \backslash U$. If $\maG$ is a Fredholm groupoid, we have:
\begin{enumerate}[(i)]

\item The canonical projection induces an isomorphism
  $C_r\sp{\ast}(\maG)/C_r\sp{\ast}(\maG_{U}) \simeq
  C_r\sp{\ast}(\maG_F)$, that is, we have the exact sequence
  \begin{equation*}
  0  \longrightarrow C\sp{\ast}_{r}(\maG_U)\cong \maK \longrightarrow C\sp{\ast}_{r}(\maG)
  \mathop{\xrightarrow{\hspace*{1cm}}}\limits^{(\rho_F)_{r}}
  C\sp{\ast}_{r}(\maG_F)  \longrightarrow 0 \,.
\end{equation*}
\item $\maR(\maG_{F})=\{\pi_x,\, x \in F\}$ is  an exhaustive set of
  representations of $C_r\sp{\ast}(\maG_F)$.
\end{enumerate}

Conversely, if $\maG \tto M$ satisfies (i) and (ii), then, for any
unital $C^{\ast}$-algebra $\mathbfPsi$ containing $C^{\ast}_r(\maG)$
as an essential ideal,
and for any $a \in \mathbfPsi $, we have that
$a$ is Fredholm on $L^2(U)$ if, and only if, $\pi_x(a)$ is invertible for
each $x \notin U$ {\bf and} the image of $a$ in
$\mathbfPsi/C^{\ast}_r(\maG)$ is invertible.
\end{theorem}

In \cite{CNQ, CNQ17}, we dubbed condition (ii) as {\em Exel's property} (for  $\maG_F$). If $\maR(\maG_{F})=\{\pi_x,\, x \in F\}$ is an exhaustive set of
representations of $C\sp{\ast}(\maG_F)$, then $\maG_{F}$ is said to have the {\em strong Exel property}. In this case, it is metrically amenable.
We will use the sufficient conditions in Theorem \ref{thm.Fredholm.Cond} in the following form:

\begin{proposition}\label{cor.Fredholm.Cond}
Let $\maG \tto M$ be a Hausdorff Lie groupoid and $U$ an open, dense, $\maG$-invariant subset  such that $\maG_U \simeq U \times U$. Let $F=M \backslash U$.
Assume $\maR(\maG_{F})=\{\pi_x,\, x \in F\}$ is an exhaustive set of
representations of $C\sp{\ast}(\maG_F)$. Then $\maG$ is Fredholm and metrically amenable.
\end{proposition}

This characterization of Fredholm groupoids, together with the properties of exhaustive families, allows us to show that large classes of groupoids are Fredholm. See for instance Corollary \ref{prop.fred} below and all the examples in Subsection \ref{sub:boundary_action_groupoids}, and more generally \cite{CNQ2017}.

\begin{proof}
Condition (ii)  in Theorem \ref{thm.Fredholm.Cond}  holds by assumption.
If $\maR(\maG_{F})$ is a strictly spectral set of
representations of $C\sp{\ast}(\maG_F)$ then, by definition, the reduced and full norms on $\cC_c(\maG_F)$ coincide, hence $\maG_{F}$ is metrically amenable. It follows from the exact sequences  \eqref{renault.exact_item1} and \eqref{renault.exact_item3}, since $\maG_{U} \simeq U \times U$ is metrically amenable, that $\maG$ is metrically amenable and that condition (i) in Theorem \ref{thm.Fredholm.Cond} also holds.
Taking the unitalization $\mathbfPsi:=\left(C^{\ast}(\maG)\right)^{+}$, we have then that $\maG$ is Fredholm.
\end{proof}

Representations are extended to matrix algebras in the
obvious way, which allows us to treat operators on vector bundles.

\begin{remark}
The notion of exhaustive family can be linked to that of $EH$-amenability and to the Effros-Hahn conjecture \cite{CNQ17, nistorPrudhon}.
Let $\maG\tto F$ be an $EH$-amenable locally compact groupoid.  Then
the family of regular representations $\{\pi_y, y \in F\}$ of $C\sp{\ast}(\maG)$
is exhaustive, hence strictly spectral.
Hence if $U$ is a dense invariant subset such that $\maG_{U}$ is the pair groupoid and $\maG_{F}$ is $EH$-amenable, then $\maG$ is Fredholm. Combining with the proof of the generalized EH conjecture
\cite{ionescuWilliamsEHC, renault87, renault91} for amenable,
Hausdorff, second countable groupoids, we get a set of sufficient conditions for $\maG$ to be Fredholm.
\end{remark}

\begin{example}\label{expl.transfgroupoidFredholm}
Let $\overline{\maH}=[0,\infty] \rtimes (0,\infty)$ be the transformation groupoid
with the action of $(0,\infty)$ on $[0,\infty ]$ by dilation,
(that is, $\overline{\maH}$ is the extension of the groupoid in Example \ref{ex:action_groupoid} to the one point compactification of $[0,\infty)$). Then $\overline{\maH}$ is Fredholm. Indeed, it is clear that $(0,\infty)\subset [0,\infty]$ is an invariant open dense subset, and $\overline{\maH}|_{(0,\infty)} \simeq (0,\infty)^2$, the pair groupoid of $(0,\infty)$.
Then $F=\{0, \infty\}$, $\overline{\maH}_{F}\cong (0,\infty)\sqcup (0,\infty)$, the disjoint union of two amenable Lie groups, and $C\sp{\ast}(\overline{\maH}_{F})\cong \maC_{0}(\RR^{+}) \oplus \maC_{0}(\RR^{+})$. Hence $\overline{\maH}_{F}$ has Exel's property (the regular representations at $0$ and $\infty$ are induced from the regular representation of the group, which is just convolution).
So $\overline{\maH}$ is Fredholm.

Note that if we have a convolution operator $K$ on the abellian group $(0,\infty)$, for instance the double layer potential operator, we can identify $K$ with a family of convolution operators $K_x$, $x\in (0,\infty)$ (we use the fact that the action groupoid $(0,\infty) \rtimes (0,\infty)$ is isomorphic to the pair groupoid of $(0,\infty)$.)
Since each $K_x$ is a convolution operator, we can
always extend by continuity the family $K_x$, $x\in (0,\infty)$ to the family $K_x$, $x\in [0,\infty]$ (two endpoints included).
In this way, we identify $K$ with an operator on the groupoid $[0,\infty] \rtimes (0,\infty)$ (note however, that the reduced support of $K$ may not be compact, so  it might not be a pseudodifferential operator on the groupoid $\overline{\maH}$, according to our previous definition).
\end{example}

In the next example, we study an important class of Lie
groupoids for which the set  of regular
representations is an exhaustive set of representations of
$\Cs{\maG}$. The point is that
 locally, our groupoid is the product of a group $G$ and a space, so its
$C^*$-algebra  is of the form $C^*(G) \otimes \maK$, where
$\maK$ are the compact operators.
See \cite[Proposition 3.10]{CNQ2017} for a complete proof.

\begin{example}\label{ex.Exel}
Let $\maH \tto B$ be a locally trivial bundle of groups, so $d = r$,
with fiber a locally compact group $G$. Then $\maH$ has Exel's property, that is, the set of regular representations $\maR(\maH)$ is exhaustive for $C^{\ast}_r(\maH)$, since any irreducible representation of $C_{r}^{\ast}(\maH)$ factors through evaluation at
$\maH_{x}\cong G$, and
the regular representations of $\maH$  are obtained from the regular representation of $G$.
It is exhaustive for the full algebra $C^{\ast}(\maH)$ if, and only if,
the group $G$ is amenable.

More generally, let $f : M \to B$ be a continuous surjective map. Then $\maG=f\pullback(\maH)$ is a locally compact groupoid with a Haar system that also has Exel's property, and $\maR(\maG)$ is exhaustive for $C^{\ast}(\maG)$ if, and only if,
the group $G$ is amenable (note that $G$ is isomorphic to the isotropy group $\maH_{x}^{x}$, for $x\in M$). This stems from the fact that $\maH$ and $f\pullback(\maH)$ are \emph{Morita equivalent} groupoids, hence have homeomorphic primitive spectra \cite{CNQ2017}.

\begin{remark}
In fact,  $f\pullback (\maH)$
satisfies the generalized EH conjecture, and hence it has the
weak-inclusion property. It will be EH-amenable if, and only if, the group $G$ is
amenable (see \cite{CNQ17}).
\end{remark}

\end{example}

Putting together the previous example and Proposition \ref{cor.Fredholm.Cond}, we conclude the following:
\begin{corollary}\label{prop.fred}
Let $\maG\tto M$ is a Hausdorff Lie groupoid with $U\subset M$ an open, dense, invariant subset. Set $F=M\setminus U$ and assume that we have a decomposition
$\maG_U \simeq U \times U$ and $\maG_F \simeq f\pullback(\maH)$; in particular,
\begin{equation*}
\maG = (U\times U) \sqcup f\pullback(\maH),
\end{equation*}
where $f: F\to B$ is a {tame submersion }and $\maH\tto B$ is a bundle of amenable Lie groups. Then $\maG$ is Fredholm.
\end{corollary}

Corollary \ref{prop.fred} is enough to obtain the Fredholm property for many groupoids used in applications.
Several examples can be found in \cite[Section 5]{CNQ} (see also \cite{CNQ17}). They include the $b$-groupoid modelling manifolds with poly-cylindrical ends, groupoids modelling analysis on asymptotically Euclidean spaces, asymptotically hyperbolic
spaces, and the edge groupoids. Some of these examples will be discussed in Subsection \ref{sub:boundary_action_groupoids}.

\smallskip

We consider Fredholm groupoids because of their applications
to Fredholm conditions. Let  $\Psi\sp{m}(\maG)$ be the space of order $m$,
classical pseudodifferential operators $P = (P_x)_{x \in M}$ on $\maG$, as in Subsection \ref{ss.ops.grpds}.
Recall that $P_x \in
\Psi\sp{m}(\maG_x)$, for any $x \in M$ and that $P_x=\pi_x(P)$, with $\pi_x$ the regular representation of $\G$ at $x \in M$. The operator $P$ acts on $U$ via the (injective) vector representation $\pi_0 : \Psi^m(\G) \to \maL(H^s(U),H^{s-m}(U))$ and that $L^{m}_{s}(\maG)$
is the norm closure of $\Psi^m(\maG)$ in the topology of
continuous operators $H^s(U)\to H^{s-m}(U)$.

Recall that a differential operator $P:C^\infty(U) \to C^\infty(U)$ is called \emph{elliptic} if its principal symbol $\sigma(P) \in \Gamma(T^*U)$ is invertible outside the zero-section \cite{Hor1985}. The following Fredholm condition is one of the main results of \cite{CNQ2017}.

\begin{theorem}[{\citeauthor{CNQ} \cite[Theorem 4.17]{CNQ}}] \label{thm.nonclassical2}
Let $\maG \tto M$ be a Fredholm Lie groupoid and let
$U \subset M$ be the dense, $\maG$-invariant subset such that
$\maG_{U} \simeq U \times U$.
 Let $s\in \RR$ and $P \in L^m_s(\maG) \supset \Psi\sp{m}(\maG)$.
We have
\begin{multline*} 
	P : H^s(U) \to H^{s-m}(U)\
        \mbox{ is Fredholm} \ \ \Leftrightarrow \ \ P \mbox{ is
          elliptic and }\\
	\ P_{x} : H^s(\maG_x) \to H^{s-m}(\maG_x) \ \mbox{ is
          invertible for all } x \in F := M \smallsetminus U \,.
\end{multline*}
\end{theorem}

\begin{proof}
  This theorem is proved by considering $a := (1 + \Delta)\sp{(s-m)/2} P
(1 + \Delta)\sp{-s/2}$, which belongs to the $C^*$-algebra $\overline{\Psi}(\maG) \, =:\, L^{0}_{0}(\maG)$
by the results in \cite{LMN, LN}. Since $\overline{\Psi}(\G)$ contains $C^*_r(\G)$ as an essential ideal, the conclusion follows from Theorem \ref{thm.Fredholm.Cond}. See \cite{CNQ2017} for more details.
\end{proof}

Theorem \ref{thm.nonclassical2} extends straightforwardly to operators acting between sections of vector bundles. The operators $P_x$, for $x \in M \setminus U$, are called \emph{limit operators} of $P$. Note that $P_x$ is invariant under the action of the isotropy group $\G_x^x$ on the fiber $\G_x$. Similar characterizations of Fredholm operators were obtained in different contexts in \cite{Mel1993,Sch1991,GI2006,DLR2015,Dau1988}, to cite a few examples.

\vspace{0.3cm}
\section{Boundary Action Groupoids}\label{sec:gluing_groupoids}

We describe in this section a procedure for gluing locally compact groupoids. This extends a construction of Gualtieri and Li that was used to classify the Lie groupoids integrating certain Lie algebroids \cite{GL2014} (see also \cite{Nis2015}).

\subsection{The gluing construction}
\label{sub:gluing_groupoids}

Let $X$ be a locally compact Hausdorff space, covered by a family of open sets $(U_i)_{i\in I}$. Recall that, if $\G \tto X$ is a locally compact groupoid and $U \subset X$ an open set, then the \emph{reduction} of $\G$ to $U$ is the open subgroupoid $\G|_U := \G_U^U = d^{-1}(U) \cap r^{-1}(U)$.

Now, for each $i \in I$, let $\G_i \rightrightarrows U_i$ be a locally compact groupoid with domain $d_i$ and range $r_i$. Assume that we are given a family of isomorphisms between all the reductions
\begin{equation*}
  \phi_{ji} : \G_i|_{U_i\cap U_j} \to \G_j|_{U_i\cap U_j},
\end{equation*}
such that $\phi_{ij}=\phi_{ji}^{-1}$ and $\phi_{ij}\phi_{jk} = \phi_{ik}$ on the common domains.
Our aim is to glue the groupoids $\G_i$ to build a groupoid $\G \st X$ such that, for all $i\in I$,
\begin{equation*}
  \G|_{U_i} \simeq \G_i.
\end{equation*}

As a topological space, the groupoid $\G$ is defined as the quotient
\begin{equation} \label{eq:quotient_space}
  \G = \bigsqcup_{i\in I} \G_i \big/ \sim,
\end{equation}
where $\sim$ is the equivalence relation generated by $g \sim \phi_{ji}(g)$, for all $i,j \in I$ and $g \in \G_i$. Since each $\G_i$ is a locally compact space, the space $\G$ is also locally compact for the quotient topology. If $g\in\G$ is the equivalence class of a $g_i \in \G_i$, we define
\begin{equation*}
  d(g) = d_i(g_i) \qquad \text{and} \qquad r(g) = r_i(g_i).
\end{equation*}
Because the groupoids $\G_i$ are isomorphic on common domains $U_i \cap U_j$, for $i,j \in I$, this definition is independent on the choice of the representative $g_i$. The unit $u: X \to \G$ and inverse maps are defined in the same way. Therefore, the subsets $\G|_{U_i} = r^{-1}(U_i) \cap d^{-1}(U_i)$ are well defined, for each $i \in I$.

\begin{lm} \label{lm:pi_i_homeo}
  For each $i\in I$, the quotient map $\pi_i : \G_i \to \G$ induces an homeomorphism (of topological spaces)
  \begin{equation*}
    \pi_i : \G_i \to \G|_{U_i}.
  \end{equation*}
\end{lm}
\begin{proof}
  The topology on $\G$ is the coarsest one such that each quotient map $\pi_i$ is open and continuous, for every $i \in I$. Moreover, for any $i \in I$, the definition of the equivalence relation $\sim$ in Equation \eqref{eq:quotient_space} implies that $\pi_i$ is injective. Therefore, the map $\pi_i$ is a homeomorphism onto its image, which is obviously contained in $\G|_{U_i}$.

  To prove that $\pi_i(\G_i) = \G|_{U_i}$, let $g \in \G|_{U_i}$ be represented by an element $g_j \in \G_j$, for $j \in I$. Then $g_j \in \G_j|_{U_i\cap U_j}$, which is isomorphic to $\G_i|_{U_i\cap U_j}$ through $\phi_{ij}$ : thus $g$ also has a representative in $\G_i$. This shows that $\pi_i(\G_i) =\G|_{U_i}$.
\end{proof}

In particular, Lemma \ref{lm:pi_i_homeo} implies that the structural maps $d,r,u$ and $\iota$ are continuous and that the domain and range maps $d,r : \G \to X$ are open. With Remark \ref{rem:defn_groupoid_better} in mind, the only missing element to have a groupoid structure on $\G$ is a well-defined product. Therefore, define the set of composable arrows by
\begin{equation*}
  \G^{(2)} = \{(g,h) \in \G, d(g) = r(g)\}.
\end{equation*}
A problem is that there are a priori no relation between the two groupoids $\G_i$ and $\G_j$, for $i \neq j$. Thus, if  $(g_i,g_j)\in\G^{(2)}$ with $g_i \in \G_i$ and $g_j \in \G_j$, then there is a priori no obvious way of defining the product $g_ig_j$ in $\G$. A way around this issue is to introduce a \enquote{gluing condition}, so that any composable pair $(g,h) \in \G^{(2)}$ is actually contained in a single groupoid $\G_k$, for a $k \in I$.

\begin{defn}
  \label{defn:weak_gluing_condition}
  We say that a family $(\G_i \st U_i)_{i\in I}$ of locally compact groupoids satisfy the \emph{weak gluing condition} if for every composable pair $(g,h) \in \G^{(2)}$, there is an $i \in I$ such that both $g$ and $h$ have a representative in $\G_i$.
\end{defn}
An equivalent statement of Definition \ref{defn:weak_gluing_condition} is to say that the family $(\G_i^{(2)})_{i \in I}$ is an open cover of the space of composable arrows $\G^{(2)}$.

\begin{lm} \label{lm:groupoid_struct_on_the_gluing}
  Assume that the family $(\G_i)_{i\in I}$ satisfy the weak gluing condition. Then there is a unique groupoid structure on
  \begin{equation*}
    \G = \bigsqcup_{i\in I} \G_i / \sim
  \end{equation*}
  such that the projection maps $\pi_i :\G_i \to \G|_{U_i}$ are isomorphisms of locally compact groupoids, for every $i \in I$.
\end{lm}
\begin{proof} 
  Let $(g,h) \in \G^{(2)}$ be a composable pair. The weak gluing condition implies that there is an $i \in I$ such that $g$ and $h$ have representatives $g_i$ and $h_i$ in $\G_i$. We thus define the product $gh$ as the class of $g_{i}h_{i}$ in $\G$, and we check at once that this does not depend of a choice of representative for $g$ and $h$. Lemma \ref{lm:pi_i_homeo} and the definition of the structural maps on $\G$ imply that each $\pi_i : \G_i \to \G|_{U_i}$ is an isomorphism of locally compact groupoids, for each $i \in I$.

To show the uniqueness of the groupoid structure on $\G$, let us assume conversely that each map $\pi_i : \G_i \to \G|_{U_i}$ is a groupoid isomorphism. Since the reductions $(\G|_{U_i})_{i\in I}$ cover $\G$, the domain, range, identity and inverse maps of $\G$ are prescribed by those of each $\G_i$. Moreover, the weak gluing condition implies that, for each composable pair $(g,h) \in \G^{(2)}$, both $g$ and $h$ lie in a same reduction $\G|_{U_i}$. Thus the product on $\G$ is also determined by those of each groupoid $\G_i$, for $i \in I$.
\end{proof}

\begin{defn} \label{defn:glued_groupoid}
  The groupoid $\G$ of Lemma  \ref{lm:groupoid_struct_on_the_gluing} defines the \emph{gluing} (or \emph{glued groupoid}) of a family of locally compact groupoids $(\G_i)_{i\in I}$ satisfying the gluing condition. We denote it
  \begin{equation*}
     \G = \bigcup_{i \in I} \G_i,
  \end{equation*}
  when there is no ambiguity about the family of isomorphisms $(\phi_{ij})_{i,j}$ involved.
\end{defn}

\begin{rem}
 The glued groupoid can also be defined by a universal property. Assume we only have two groupoids $\G_1 \st U_1$ and $\G_2 \st U_2$, and let $\G_{12} \eqdef \G_1|_{U_1\cap U_2} \simeq \G_2|_{U_1\cap U_2}$. Then $\G = \G_1 \cup \G_2$ is the \emph{pushout} of the inclusions morphisms $j_i: \G_{12} \hookrightarrow \G_i$, for $i = 1,2$. It is the ``smallest'' groupoid such that there is a commutative diagram

$$
\xymatrix{
    \maG & \maG_2 \\
    \maG_1 & \maG_{12}.
    \ar_{} "1,2";"1,1"
    \ar_{} "2,1";"1,1"
    \ar^{j_2}"2,2";"1,2"
    \ar^{j_1}"2,2";"2,1"}
$$

When we have a general family $(\G_i)_{i\in I}$ satisfying the gluing condition, the glued groupoid can similarly be defined as the colimit relative to the inclusions $\G|_{U_i \cap U_j} \hookrightarrow \G_i$, for all $i,j \in I$.
\end{rem}

\begin{rem}
  It is possible for a family $(\G_i)_{i\in I}$ to satisfy the weak gluing condition, even though there is a pair $(\G_{i_0},\G_{j_0})$ that do not satisfy the gluing condition, for some $i_0,j_0 \in I$. For instance, let $X$ be a locally compact, Hausdorff space and $U_1, U_2$ two distinct open subsets in $X$ with non-empty intersection $U_{12}$. Let
  \begin{equation*}
    \G_0 = X \times X, \quad \G_1 = U_1 \times U_1 \quad \text{and} \quad \G_2 = U_2 \times U_2
  \end{equation*}
  be pair groupoids over $X$, $U_1$ and $U_2$ respectively. The family $(\G_0,\G_1,\G_2)$ satisfies the weak gluing condition of Definition \ref{defn:weak_gluing_condition}, and may be glued to obtain the groupoid $\G = X \times X = \G_0$. However, the pair $(\G_1,\G_2)$ does not satisfy the weak gluing condition.
\end{rem}

\begin{lm} \label{lm:gluing_Hausdorff}
  Let $(\G_i)$ be a family of groupoids satisfying the weak gluing condition. If each $\G_i$, for $i \in I$, is a Hausdorff groupoid, then the gluing $\G = \bigcup_{i \in I}\G_i$ is also Hausdorff.
\end{lm}
\begin{proof}
  Let $g,h \in \G$. There are two cases.
  \begin{itemize}
    \item Assume $d(g) = d(h)$ and $r(g) = r(h)$. Then, because of the gluing condition, there is an $i \in I$ such that $g$ and $h$ are both in the Hausdorff groupoid $\G|_{U_i}$.
    \item Otherwise, either $d(g) \neq d(h)$ or $r(g) \neq r(h)$. Let us assume the former. Then, since $X$ is Hausdorff, there are open sets $U,V \subset X$ such that $d(g) \in U$, $d(h) \in V$ and $U\cap V = \emptyset$. Thus $g \in \G_U$ and $h \in \G_V$, which are disjoint open subsets of $\G$. \qedhere
  \end{itemize}
\end{proof}

We also introduce the \emph{strong gluing condition}, which is often easier to check.
\begin{defn}
  \label{defn:strong_gluing_condition}
  We say that the family $(\G_i \st U_i)_{i\in I}$ of locally compact groupoids satisfy the \emph{strong gluing condition} if, for each $x\in X$, there is an $i_x \in I$ such that
      \begin{equation*}
        \G_i \cdot x \subset U_{i_x}
      \end{equation*}
    for all $i \in I$.
\end{defn}
In other words, the orbit of a point through the action of $\G$ should always be induced by a single element of the family $(\G_i)_{i \in I}$.

\begin{lm}
  Let $(\G_i)_{i \in I}$ be a family of groupoids which satisfies the strong gluing condition. Then the family $(\G_i)_{i \in I}$ also satisfies the weak gluing condition.
\end{lm}
\begin{proof}
  Let $(g,h) \in \G^{(2)}$, and assume that $g$ has a representative $g_i \in \G_i$ and $h$ a representative $h_j \in \G_j$. Let $x = d(g) = r(h)$. The gluing condition implies that there is an $i_x \in I$ such that $\G_i \cdot x \subset U_{i_x}$ and $\G_j \cdot x \subset U_{i_x}$. Thus $r_i(g_i) \in U_{i_x}$, so $g_i \in \G_i|_{U_i\cap U_{i_x}}$. But there is an isomorphism
\begin{equation*}
  \phi_{{i_x}i} : \G_i|_{U_i\cap U_{i_x}} \to \G_{i_x}|_{U_i\cap U_{i_x}}
\end{equation*}
so that $g$ actually has a representative $g_{i_x}$ in $\G_{i_x}$. The same arguments show that $h$ also has a representative $h_{i_x} \in \G_{i_x}$.
\end{proof}

We conclude this subsection with a condition for which a groupoid $\G \st X$ may be written as the gluing of its reductions. This definition was introduced by Gualtieri and Li for Lie algebroids \cite{GL2014}.

\subsection{Gluing Lie groupoids}
\label{sub:orbit_covers_and_lie_algebroids}
Let $M$ be a manifold with corners, and $(U_i)_{i\in I}$ an open cover of $M$. Let $(\G_i)_{i \in I}$ be a family of \emph{Lie groupoids} satisfying the weak gluing condition of Definition \ref{defn:weak_gluing_condition}. Assume that the morphisms $\phi_{ji} : \G_i|_{U_i \cap U_j} \to \G_j|_{U_i \cap U_j}$ are Lie groupoid morphisms, and let $\G \eqdef \bigcup_{i\in I} \G_i$ be the glued groupoid over $M$.
\begin{lm}
  If each $\G_i$, for $i \in I$, is a Lie groupoid, then there is a unique Lie groupoid structure on $\G$ such that $\pi_i : \G_i \to \G|_{U_i}$ is an isomorphism of Lie groupoids, for all $i \in I$.
\end{lm}
\begin{proof}
  By Definition \ref{defn:glued_groupoid}, the reductions $\G|_{U_i}\simeq \G_i$, for $i \in I$, provide an open cover of $\G$. Since each $\G_i$ is a Lie groupoid, and all $\phi_{ij}$ are smooth, this induces a manifold structure on $\G$. Each structural map of $\G$ coincides locally with a structural map of one of the groupoids $\G_i$, hence is smooth. This gives the Lie groupoid structure.
\end{proof}
\begin{rem}
  A similar statement holds when each $\G_i$ is a continuous family groupoids, for all $i \in I$: then $\G$ is also a continuous family groupoid \cite{LMN2000,Pat2001}.
\end{rem}

To specify the Lie algebroid of $\G$, we need first study the gluing of Lie algebroids. For each $i \in I$, let $A_i \to U_i$ be a Lie algebroid. Assume that there are Lie algebroid isomorphisms $\psi_{ij} : A_i|_{U_i \cap U_j} \to A_j|_{U_i \cap U_j}$ covering the identity, such that $\psi_{ij}^{-1} = \psi_{ji}$ and $\psi_{ij}\psi_{jk} = \psi_{ik}$ on common domains. As vector bundles, the family $(A_i)_{i\in I}$ is in particular a family of groupoids that satisfies the strong gluing condition of Definition \ref{defn:strong_gluing_condition} (the orbit of any $x \in M$ is reduced to $\{x\}$). Thus, the gluing $A = \bigcup_{i \in I} A_i$ is a smooth vector bundle on $M$, with inclusion maps $\pi_i : A_i \hookrightarrow A$.

\begin{lm}
  There is a unique Lie algebroid structure on $A = \bigcup_{i \in I} A_i$ such that each map $\pi_i : A_i \to A$ is a morphism of Lie algebroids.
\end{lm}
\begin{proof}
  By definition, the Lie algebroids $A|_{U_i} \simeq A_i$, for all $i \in I$, provide an open cover of $A$. Let $X,Y \in \Gamma(A)$, and define $[X,Y] \in \Gamma(A)$ by
  \begin{equation*}
    [X,Y]|_{U_i} \eqdef [X|_{U_i},Y|_{U_i}]_i,
  \end{equation*}
  where $[.,.]_i$ is the Lie bracket on $A_i$. Since $A_i|_{U_i \cap U_j}$ and $A_j|_{U_i \cap U_j}$ are isomorphic as Lie algebroid, the section $[X,Y]$ is well-defined on $U_i \cap U_j$, for all $i,j \in I$. This defines the Lie bracket on $\Gamma(A)$. The anchor map is similarly defined as $\rho(X)|_{U_i} \eqdef \rho_i(X|_{U_i})$, with $\rho_i$ the anchor map of $A_i$. Because the family $(A_i)_{i\in I}$ covers $A$, this it is the unique Lie algebroid structure on $A$ such that each map $\pi : A_i \to A|_{U_i}$ is a Lie algebroid isomorphism.
\end{proof}

\begin{lm} \label{lm:gluing_algebroids}
  Let $(\G_i\st U_i)_{i \in I}$ be a family of Lie groupoids satisfying the gluing condition, with isomorphisms $\phi_{ij} : \G_j|_{U_i\cap U_j} \to \G_i|_{U_i\cap U_j}$. The Lie algebroid of the resulting glued groupoid $\G = \bigcup_{i \in I} \G_i$ is isomorphic to the gluing of the family $(\A\G_i)_{i \in I}$, with Lie algebroid isomorphisms $(\phi_{ij})_* : \A\G_i|_{U_i \cap U_j} \to \A\G_j|_{U_i \cap U_j}$, for $i,j \in I$.
\end{lm}
\begin{proof}
  By definition of the quotient maps $\pi_i : \G_i \to \G$, the map $\pi_j^{-1} \circ \pi_i$ coincides with the isomorphism $\phi_{ji}: \G_i|_{U_i \cap U_j}\to\G_j|_{U_i \cap U_j}$, for all $i,j \in I$. Let $\xi \in A\G_i|_{U_i \cap U_j}$. Then
  \begin{equation} \label{eq:Lie_algebroid_quotient}
    (\pi_i)_*(\xi) = (\pi_j)_* \circ (\pi_j^{-1}\circ\pi_i)_*(\xi) = (\pi_j)_* \circ (\phi_{ji})_*(\xi) \in \A\G|_{U_i \cap U_j}
  \end{equation}

  Let $\Psi : \bigsqcup_{i\in I} \A\G_i \to \A\G$ bet the map given by $\Psi(\xi) \eqdef (\pi_i)_*(\xi)$, whenever $\xi \in \A\G_i$. Equation \eqref{eq:Lie_algebroid_quotient} implies that $\Psi$ induces a map from the quotient $A =\bigcup_{i \in I} \A\G_i$, which is the glued algebroid, to $\A\G$. Each map $\pi_i : \G_i \to \G$ gives an isomorphism $(\pi_i)_* : \A\G_i \to \A\G|_{U_i}$, so $\Psi : A \to \A\G$ is also a Lie algebroid isomorphism.
\end{proof}

\subsection{Boundary action groupoids}
\label{sub:boundary_action_groupoids}

Our aim is to study Fredholm conditions for algebras of differential operators generated by Lie groupoids $\G \st M$. To this end, we define the class of \emph{boundary action groupoids}, which are obtained by gluing reductions of action groupoids. We will show that many examples of groupoids arising in analysis on open manifold belong to this class, and obtain Fredholm condition for the associated differential operators.

Recall that gluing conditions were discussed in Subsection \ref{sub:gluing_groupoids}.

\begin{defn} \label{defn:boundary_action_groupoids}
  A Lie groupoid $\G\dr M$ is \emph{a boundary action groupoid} if
  \begin{enumerate}
    \item \label{it:defn_bag_1} there is an open dense $\G$-invariant subset $U \subset M$ such that $\G_U \simeq U \times U$;
    \item \label{it:defn_bag_2} there is an open cover $(U_i)_{i\in I}$ of $M$ such that for all $i \in I$, we have a Hausdorff manifold $X_i$, a Lie group $G_i$ acting smoothly on $X_i$ on the right and an open subset $U_i' \subset X_i$ diffeomorphic to $U_i$ satisfying
      \begin{equation*}
        \G|_{U_i} \simeq (X_i \rtimes G_i)|_{U'_i};
      \end{equation*}
    \item \label{it:defn_bag_3} the family of groupoids $(\G|_{U_i})_{i \in I}$ satisfy the weak gluing condition, with the obvious identifications of $\G|_{U_i}$ and $\G|_{U_j}$ with $\G|_{U_i\cap U_j}$ over common domains.
  \end{enumerate}
\end{defn}

In other words, boundary action groupoids are groupoids that are obtained by gluing reductions of action groupoids, and that are simply the pair groupoid over a dense orbit. Note that, as an open dense $\G$-orbit in $M$, the subset $U$ in Definition \ref{defn:boundary_action_groupoids} is uniquely determined by $\G$.

\begin{ex} \label{ex:bag_pair_groupoid}
  If $M_0$ is a smooth manifold (without corners), then the pair groupoid $\G = M_0 \times M_0$ is a boundary action groupoid. Indeed, for any triple $(x,y,z) \in M_0^3$, we can choose an open subset $U_{x,y,z} \subset M_0$ that contains $x$, $y$, and $z$ and is such that $U_{x,y,z}$ is diffeomorphic to an open subset $U'_{x,y,z} \subset \R^n$ (just choose three disjoint, relatively compact coordinate charts near each point $x$, $y$ and $z$ and take $U_{x,y,z}$ to be their disjoint union). Then
  \begin{equation*}
    \G|_{U_{x,y,z}} \simeq (\R^n \times \R^n)|_{U'_{x,y,z}} \simeq (\R_n \rtimes \R^n)|_{U'_{x,y,z}},
  \end{equation*}
  where $\R^n$ acts on itself by translation. Moreover the family of groupoids $(\G|_{U_{x,y,z}})$, for $x,y,z \in M_0$, satisfy the weak gluing condition: for any composable pair $(x,y)$ and $(y,z)$ in $\G$, both $(x,y)$ and $(y,z)$ are contained in $\G|_{U_{x,y,z}}$. This shows \eqref{it:defn_bag_2} and \eqref{it:defn_bag_3} from Definition \ref{defn:boundary_action_groupoids}, whereas \eqref{it:defn_bag_1} is trivially satisfied.
\end{ex}

Other practical examples will be introduced in Subsection \ref{sub:examples} below. One of the main points of this definition is to have a good understanding of how $\G_{U}$ and $\G_{M\setminus U}$ are glued together near the boundary. In particular:
\begin{lm} \label{lm:boundary_action_Hausdorff}
  Boundary action groupoids are Hausdorff.
\end{lm}
\begin{proof}
  We keep the notations of Definition \ref{defn:boundary_action_groupoids} above. Note that all $(X_i \rtimes G_i)|_{U'_i}$ are Hausdorff groupoids (as subsets of the Hausdorff spaces $X_i\times G_i$). Since the groupoids $(\G|_{U_i})_{i\in I}$ satisfy the weak gluing condition, the groupoid $\G$ is obtained by gluing Hausdorff groupoids. The result then follows from Lemma \ref{lm:gluing_Hausdorff}.
\end{proof}

Lemmas \ref{lm:bag_product} to \ref{lm:bag_attaching_ends} give some possible combinations of boundary action groupoids that preserve the local structure.

\begin{lm} \label{lm:bag_product}
  Let $\G \dr M$ and $\maH \dr N$ be boundary action groupoids. Then $\G\times\maH \dr M \times N$ is a boundary action groupoid.
\end{lm}
\begin{proof}
  First, let $U, V$ be the respective open dense orbits of $\G$ and $\maH$. Then $U \times V$ is an open dense orbit for $\G \times \maH$, and $(\G\times\maH)_{U\times V}$ is the pair groupoid $(U\times V)^2$. Secondly, let $(U_i)_{i\in I}$ and $(V_j)_{j \in J}$ be respective open covers of $M$ and $N$ such that we have isomorphisms
  \begin{equation*}
    \G|_{U_i} \simeq (X_i \rtimes G_i)_{U'_i} \qquad \text{and} \qquad \maH|_{V_j} \simeq (Y_j \rtimes H_j)|_{V'_j}
  \end{equation*}
and both families $(\G|_{U_i})_{i \in I}$ and $(\maH|_{V_j})_{j \in J}$ satisfy the weak gluing condition. Then the family $\{(\G\times\maH)|_{U_i \times V_j}\}$, for  $i\in I$ and  $j\in J$, satisfy the weak gluing condition over $M \times N$ and we have
\begin{equation*}
  (\G\times\maH)|_{U_i\times V_j} \simeq (X_i \times Y_j)\rtimes(G_i \times H_j),
\end{equation*}
for all $i \in I$ and $j \in J$, where the action of $G_i \times H_j$ is the product action.
\end{proof}

\begin{lm} \label{lm:bag_reduction}
  Let $\G\dr M$ be a boundary action groupoid and $V$ an open subset of $M$. Then $\G|_V$ is a boundary action groupoid.
\end{lm}

\begin{proof}
  Let $U$ be the unique open dense orbit of $\G$. Then $U\cap V$ is the unique open dense orbit of $\G|_V$, and $(\G|_{U\cap V}) \simeq (U\times U)|_{U\cap V}$ is the pair groupoid of $U \cap V$. Moreover, there is an open cover $(U_i)_{i \in I}$ of $M$ with isomorphisms
  \begin{equation*}
    \G|_{U_i} \simeq (X_i \rtimes G_i)|_{U'_i}
  \end{equation*}
  for all $i \in I$, and such that the family $(\G|_{U_i})_{i \in I}$ satisfy the weak gluing condition. For all $i \in I$, let $V_i = U_i \cap V$ and $V'_i$ be the image of $V_i$ in $U'_i$. The weak gluing condition imply that, for any pair $(g,h)$ of composable arrows in $\G|_V$, there is an $i \in I$ such that $g,h$ are both in $\G|_{U_i}$. Then $g,h \in \G|_{V_i}$, which shows that the family $(\G|_{V_i})_{i \in I}$ satisfy the weak gluing condition. Finally, we have isomorphisms
  \begin{equation*}
    \G|_{V_i} = \G|_{U_i \cap V} \simeq (X_i \rtimes G_i)|_{V'_i}
  \end{equation*}
  for all $i \in I$, which concludes the proof.
\end{proof}

\begin{lm} \label{lm:bag_attaching_ends}
  Let $M$ be a manifold with corners, and assume that we have two open subsets $U,V \subset M$ such that
  \begin{enumerate}
    \item the set $U$ is dense in $M$ and $M = U \cup V$,
    \item there is a boundary action groupoid $\maH \dr V$ whose unique open dense orbit is $U \cap V$.
  \end{enumerate}
  Then the glued groupoid $\G = \maH \cup (U\times U)$ over $M$ is a boundary action groupoid.
\end{lm}

Lemma \ref{lm:bag_attaching_ends} should be thought as a way of \enquote{attaching ends} to a pair groupoid, which will model the geometry at infinity.

\begin{proof}
  First, the definition of boundary action groupoids gives $\maH_{U\cap V} \simeq (U \cap V)^2$, so that $U\times U$ and $\maH$ are isomorphic over $U \cap V$. The pair $(\maH,U \times U)$ satisfies the strong gluing condition (the $\G$-orbit of any point in $M$ is either $U$ or contained in $V \setminus U$), so the gluing has a well-defined groupoid structure.

  It follows from the properties of the gluing (Lemma \ref{lm:groupoid_struct_on_the_gluing}) that $U$ is the unique open dense $\G$-orbit in $M$, and that $\G_U \simeq U\times U$. We know from Example \ref{ex:bag_pair_groupoid} that $U \times U$ is a boundary action groupoid. Therefore, there is an open cover $(U_i)_{i\in I}$ of $U$ and an open cover $(V_j)_{j \in J}$ of $V$ with isomorphisms
  \begin{equation} \label{eq:local_iso}
    (U\times U)_{U_i} \simeq (\R^n \rtimes\R^n)|_{U'_i} \qquad \text{and} \qquad \maH|_{V_j} \simeq (X_j \rtimes G_j)_{V'_j},
  \end{equation}
  and such that the respective families of reductions satisfy the weak gluing condition. Because $\G|_V \simeq \maH$ and $\G_U \simeq U\times U$, the isomorphisms of Equation \eqref{eq:local_iso} also hold for the reductions $\G|_{U_i}$ and $\G|_{V_j}$. Besides, any two composable arrows for $\G$ are either in $\G_U$ or in $\G|_V$, so the family $(\G|_{U_i})_{i\in I} \cup (\G|_{V_j})_{j \in J}$ also satisfies the weak gluing condition.
\end{proof}

\subsection{Examples}
\label{sub:examples}

We will show here that many groupoids occuring in the study of analysis on singular manifolds are boundary action groupoids. We will explain in Subsection \ref{sub:fredholm_groupoids} how this class of groupoids allows to obtain Fredholm conditions for many interesting differential operators. Our examples are based on the following result:

\begin{thm} \label{thm:lie_manifolds_maximal_integration}
  Let $M$ be a manifold with corners and $M_0 \eqdef M \setminus \d M$. Let $A \to M$ be a Lie algebroid, such that the anchor map $\rho$ induces an isomorphism $A|_{M_0} \simeq TM_0$ and $\rho(A|_{F}) \subset TF$ for any face $F$ of $M$. Then there is a unique Lie groupoid $\G \st M$ integrating $A$, such that $\G_{M_0} \simeq M_0 \times M_0$ and $\G_{\d M}$ is $d$-simply-connected.
\end{thm}
\begin{proof}
  The existence of such a groupoid has been proven by Debord \cite{Deb2001} and Nistor \cite{Nis2000}. If $\maH$ is another groupoid satisfying the assumptions of Theorem \ref{thm:lie_manifolds_maximal_integration}, then $\maH_{\d M}$ and $\G_{\d M}$ are $d$-simply-connected integrations of $A|_{\d M}$, so Theorem \ref{thm:Lie2} states that they are isomorphic. The main result in \cite{Nis2000} implies that $\G$ is then isomorphic to $\maH$.
\end{proof}

The groupoid $\G$ in Theorem \ref{thm:lie_manifolds_maximal_integration} will be called the \emph{maximal integration} of $A$. Based on this theorem, we give several examples of boundary action groupoids which occur naturally in the context of analysis on open manifolds : see \cite{CNQ2017} for more details.

\begin{ex}[Zero-groupoid] \label{ex:lie_struct_0}
  Consider $G_n \eqdef (0,\infty) \ltimes\R^{n-1}$, where $(0,\infty)$ acts by dilation on $\R^{n-1}$. The right action of $G_n$ upon itself extends uniquely to an action on $X_n \eqdef [0,\infty) \times \R^{n-1}$, by setting
  \begin{equation*}
    (x_1,\ldots,x_n) \cdot (t,\xi_2,\ldots,\xi_n) = (tx_1 , x_2 + x_1 \xi_2, \ldots, x_n + x_1 \xi_n).
  \end{equation*}
  The Lie algebra of fundamental vector fields for this action (recall Example \ref{ex:action_groupoid}) is the one spanned by $(x_1\d_1,\ldots,x_1\d_n)$ on $X_n$.

  To generalize this setting, let $M$ be a manifold with boundary and let $\V_0$ be the Lie algebra of all vector fields on $M$ vanishing on $\d M$. In a local coordinate system $[0,\infty)\times\R^{n-1}$ near $\d M$, we have
  \begin{equation*}
    \V_0 = \mathrm{Span}(x_1\d_1,\ldots,x_1\d_n),
  \end{equation*}
  as a $C^\infty(M)$-module.

  It follows from Serre-Swan's Theorem that there is a unique Lie algebroid $A_{0} \to M$ such that the anchor map induces an isomorphism $\Gamma(A_{0}) \simeq \V_{0}$. The \emph{zero-groupoid} $\G_{0} \dr M$ is the maximal integration of $A_{0}$, as given by Theorem \ref{thm:lie_manifolds_maximal_integration}: it is the natural space for the Schwarz kernels of differential operators that are induced by \emph{asymptotically hyperbolic} metrics on $M_0$ \cite{Mel1995}.

  \begin{thm} \label{thm:lie_struct_0_bag}
  The $0$-groupoid $\G_{0} \dr M$ is a boundary action groupoid. Moreover, for each $p \in \d M$, there is a neighborhood $U$ of $p$ in $M$, and an open set $V \subset \R_+\times\R^{n-1}$, such that
  \begin{equation*}
    \G_{0}|_U \simeq (X_n \rtimes G_n)|_V.
  \end{equation*}
\end{thm}
\begin{proof}
  For each $p \in \d M$, there is a neighborhood $U$ of $p$ in $M$ that is diffeomorphic to an open subset $V \subset \R_+\times\R^{n-1}$, through $\phi : U \to V$. The diffeomorphism $\phi$ maps $\d U$ to $\d V$, so $\phi_*(\V_0(U)) = \V_0(V)$. This implies that there is an isomorphism $A_{0}(U) \simeq A_{0}(V)$ covering $\phi$. Both groupoids $\G_{0}|_U$ and $(X_n \rtimes G_n)|_V$ are maximal integrations of $A_{0}(U) \simeq A_{0}(V)$, so Theorem \ref{thm:lie_manifolds_maximal_integration} implies that $\G_{0}|_U \simeq (X_n \rtimes G_n)|_V$.

  To prove that $\G_{0}$ is a boundary action groupoid, let $(U_i)_{i=1}^n$ be an open cover of $\d M$, such that each $\G_{0}|_{U_i}$ is isomorphic to a reduction of $X_n \rtimes G_n$, for all $i = 1,\ldots,n$. Let $U_0 = M_0$. Then $(U_i)_{i=0}^n$ is an orbit cover of $M$ that satisfies the assumptions of Definition \ref{defn:boundary_action_groupoids}.
\end{proof}
\end{ex}

\begin{ex} \label{ex:lie_struct_alpha}
  Example \ref{ex:lie_struct_0} can be slightly generalized by replacing $\V_0$ with a Lie algebra $\V\subset \Gamma(TM)$ such that, for any point $p \in \d M$, there is a $n$-tuple $\alpha \in \N^n$ and a local coordinate system $[0,\infty)\times\R^{n-1}$ near $p$ with
  \begin{equation*}
    \V = \mathrm{Span}(x_1^{\alpha_1}\d_1,\ldots,x_1^{\alpha_n}\d_n).
  \end{equation*}
  If $\alpha_1 = 1$ and every $\alpha_i \ge 1$, for $i =2,\ldots,n$, then the maximal integration $\G \st M$ of $\V$ is again a boundary action groupoid. Indeed, consider the action of $(0,\infty)$ on $\R^{n-1}$ given by
  \begin{equation*}
    t\cdot (x_2,\ldots,x_n) = (t^{\alpha_2}x_2,\ldots,t^{\alpha_n}x_n),
  \end{equation*}
  and form the semidirect product $G_\alpha = (0,\infty) \ltimes_\alpha \R^{n-1}$ given by this action. As in Example \ref{ex:lie_struct_0}, the right action of $G_\alpha$ upon itself extends uniquely to an action on $X_n \eqdef [0,\infty) \times \R^{n-1}$, by setting
  \begin{equation*}
    (x_1,\ldots,x_n) \cdot (t,\xi_2,\ldots,\xi_n) = (tx_1 , x_2 + x_1^{\alpha_2} \xi_2, \ldots, x_n + x_1^{\alpha_n} \xi_n).
  \end{equation*}
  An argument analogous to that of Theorem \ref{thm:lie_struct_0_bag} shows that $\G$ is obtained by gluing reductions of actions groupoids $X_n \rtimes G_\alpha$, for some $n$-tuples $\alpha \in \N^n$.
\end{ex}

\begin{ex}[Scattering groupoid] \label{ex:lie_struct_sc}
  Let $\S^n_+$ be the stereographic compactification of $\R^n$. Consider the action of $\R^n$ upon itself by translation, and extend it to $\S^n_+$ in the only possible way, by a trivial action on $\d \S^n_+$. The action groupoid $\G_{sc} = \S^n_+ \rtimes \R^n$ has been much studied in the literature, and is related to the study of the spectrum of the $N$-body problem in Euclidean space \cite{GI2006,MNP2017}.

  As in Example \ref{ex:lie_struct_sc}, we can generalize this setting to any manifold with boundary $M$. Let $\V_b$ be the Lie algebra of vector fields on $M$ which are tangent to the boundary, and let $x \in C^\infty(M)$ be a defining function for $\d M$. We define the Lie algebra of \emph{scattering vector fields} on $M$ as $\V_{sc} \eqdef x \V_b$. In a local coordinate system $[0,\infty) \times \R^{n-1}$ near $\d M$, we have
  \begin{equation*}
    \V_{sc} = \mathrm{Span}(x_1^2\d_1,x_1\d_2,\ldots,x_1\d_n),
  \end{equation*}
  as a $C^\infty(M)$-module. One can check that, when $M = \S^n_+$ as above, then $\V_{sc}$ is the Lie algebra of fundamental vector fields induced by the action of $\R^n$ on $\S^n_+$.

  As in Example \ref{ex:lie_struct_0}, there is a unique Lie algebroid $A_{sc} \to M$ whose sections are isomorphic to $\V_{sc}$ through the anchor map. The \emph{scattering groupoid} $\G_{sc} \dr M$ is the maximal integration of $A_{sc}$, and it generates the algebra of differential operators on $M_0$ that are induced by \emph{asymptotically Euclidean} metrics \cite{Mel1995}. The proof of Theorem \ref{thm:lie_struct_0_bag} can be adapted to this context to give:

\begin{thm}
  The scattering groupoid $\G_{sc} \dr M$ is a boundary action groupoid. Moreover, for each $p \in \d M$, there is a neighborhood $U$ of $p$ in $M$, and an open set $V \subset \S^n_+$, such that
  \begin{equation*}
    \G_{sc}|_U \simeq (\S^n_+ \rtimes \R^n)|_V.
  \end{equation*}
\end{thm}
\end{ex}

We introduce and study in Section \ref{s.LP_groupoids} another example of boundary action groupoid, used to model layer potentials methods on conical domains.

\subsection{Fredholm conditions}
\label{sub:fredholm_groupoids}

Let $\G \dr M$ be a boundary action groupoid, with $U$ its unique dense $\G$-orbit. Our aim in this Subsection is to obtain some conditions under which $\G$ is a Fredholm groupoid, as introduced in Subsection \ref{ss.fredholm}.

Recall that the algebra of pseudodifferential operators $\Psi^\infty(\G)$ was introduced in Subsection \ref{ss.ops.grpds}, together with the closure $L^m_s(\G)$ of the space of order-$m$ pseudodifferential operators in $\B(H^s(U),H^{s-m}(U))$.

\begin{thm} \label{thm:boundary_action_Fredholm}
  Let $\G \st M$ be a boundary action groupoid, and $U \subset M$ its unique dense orbit. Assume that the action of $\G$ on $F \eqdef M \setminus U$ is trivial, and that for all $x \in \d M$, the group $\G_x^x$ is amenable. Let $P$ be an operator in $L^m_s(\G)$. Then for all $s \in \R$, the operator $P : H^s(U) \to H^{s-m}(U)$ is Fredholm if, and only if:
  \begin{enumerate}
    \item $P$ is elliptic, and
    \item $P_x : H^s(\G^x_x) \to H^{s-m}(\G_x^x)$ is invertible for all $x \in F $.
  \end{enumerate}
\end{thm}
Under the assumptions of Theorem \ref{thm:boundary_action_Fredholm}, the characterization of Fredholm operators in $L^m_s(\G)$ reduces to the study of right-invariant operators $P_x$ on the amenable groups $\G_x^x$, for $x \in M \setminus U$.
It should be emphasized that, if $P$ is a geometric operator (Dirac, Laplacian\ldots) for a metric on $U$ which is \enquote{compatible} with $\G$ (in a sense made precise in \cite{LN}), then each $P_x$ is an operator of the same type induced by a right-invariant metric on the amenable group $\G_x^x$.
Theorem \ref{thm:boundary_action_Fredholm} extends straightforwardly to pseudodifferential operators acting between sections of vector bundles. 

\begin{proof}[Proof of Theorem \ref{thm:boundary_action_Fredholm}]
  First, according to Lemma \ref{lm:boundary_action_Hausdorff}, the groupoid $\G$ is Hausdorff. Let $(U_i)_{i\in I}$ be an open cover satisfying the conditions of Definition \ref{defn:boundary_action_groupoids}, and let $F_i = U_i \cap F$. Because the family $(\G|_{U_i})_{i \in I}$ satisfies the weak gluing condition over $M$ and $F$ is $\G$-invariant, the family $(\G|_{F_i})_{i \in I}$ also satisfies the weak gluing condition over $F$. In other words, the groupoid $\G_F$ is the gluing of the family $(\G|_{F_i})_{i\in I}$. Moreover, the action of $\G$ on $F$ is trivial, so each $\G|_{F_i}$ is isomorphic to $F_i\times G_i$, for every $i \in I$.

  These local trivializations show that $\G_{F}$ is a Lie group bundle over each connected component of $F$. Each $G_i$ is amenable, for $i \in I$, so we can conclude from Corollary \ref{prop.fred} that $\G$ is a Fredholm groupoid. Theorem \ref{thm:boundary_action_Fredholm} is then a consequence of Theorem \ref{thm.nonclassical2}.
\end{proof}

\begin{ex}
  The scattering groupoid $\G_{sc}$ of Example \ref{ex:lie_struct_sc} satisfies the assumptions of Theorem \ref{thm:boundary_action_Fredholm}. When $P \in \Psi^m(\G_{sc})$, the limit operators $(P_x)_{x \in \d M}$ are translation-invariant operators on $\R^n$. In that case, the operator $P_x$ is simply a Fourier multiplier on $C^\infty(\R^n)$, whose invertibility is easy to study: see \cite{CN2014}.
\end{ex}

\begin{ex}
  The $0$-groupoid $\G_0$ of Example \ref{ex:lie_struct_0}, which models asymptotically hyperbolic geometries, also satisfies the assumptions of Theorem \ref{thm:boundary_action_Fredholm}. If $P \in \Psi^m(\G_0)$, the limit operators $P_x$ are order-$m$, right-invariant pseudodifferential operators on the noncommutative groups $G_n = (0,\infty) \ltimes \R^{n-1}$.
\end{ex}

\begin{rem}
  We will show in a subsequent paper \cite{Com2018b} that a result similar to Theorem \ref{thm:boundary_action_Fredholm} holds without the assumption of a trivial action of $\G$ on $\d M$ (the proof requires a more involved study of the representations of $\G$). Thus all boundary action groupoids that are obtained by gluing actions by amenable groups are Fredholm groupoids. We believe the converse not to be true, although we are unable to provide any example of a Lie groupoid (with an open, dense orbit $U$ on which $\G_U \simeq U \times U$) that is not also a boundary action groupoid.
\end{rem}

\vspace{0.3cm}
\section{Layer Potentials Groupoids}\label{s.LP_groupoids}

In this section, we   review the construction of layer potentials groupoids for conical domains in \cite{CQ13}. In order to study layer potentials operators, which are operators on the boundary, we consider a groupoid over the desingularized boundary. Our aim is to relate this groupoid with the boundary action groupoids defined in the previous section in an explicit way, which we shall do in Section \ref{s.FredCond}.

\subsection{Conical domains and desingularization}
We begin with the definition of domains with conical points \cite{BMNZ, CQ13, MN10}.

\begin{definition}\label{domain}
Let $\Omega \subset \RR^n$, $n\geqslant 2$, be an open connected
bounded domain. We say that $\Omega$ is a \emph{domain with conical
  points} if there exists a finite number of points $\{p_1, p_2,
\cdots, p_l\} \subset \partial \Omega$, such that
\begin{enumerate}
\item[(1)] $\partial \Omega \backslash \{p_1, p_2, \cdots, p_l\}$ is smooth;
\item[(2)] for each point $p_i$, there exist a neighborhood $V_{p_i}$
  of ${p_i}$, a possibly disconnected domain $\omega_{p_i} \subset
  S^{n-1}$, $\omega_{p_i} \neq S^{n-1}$, with smooth boundary, and a
  diffeomorphism $\phi_{p_i}: V_{p_i} \rightarrow B^{n}$ such that
$$\phi_{p_i}(\Omega \cap V_{p_i})=\{rx': 0<r< 1, x'\in \omega_{p_i}\}.$$
(We assume always that $\overline{V_i}\cap \overline{V_j}=\emptyset$, for $i\neq j$, $i,j\in\{1,2,\cdots, l\}$.)
\end{enumerate}
If $\pa\Omega= \pa\overline{\Omega}$, then we say that
$\Omega$ is a \emph{domain with no cracks}. The points $p_i$, $i= 1,
\cdots, l$ are called \emph{conical points} or \emph{vertices}. If
$n=2$, $\Omega$ is said to be a {\em polygonal domain}.
\end{definition}

We shall distinguish two cases: \emph{conical domains without cracks}, $n\in\mathbb{N}$, and
\emph{polygonal domains with ramified cracks}. (Note that if $n\geq 3$ then domains with cracks have edges, and are no longer conical.)

\vspace{0.2cm}

For simplicity, we assume $\Omega$ to be a subset of $\RR^n$.

\vspace{0.1cm}

In applications to boundary value problems in $\Omega$, it is often useful to regard smooth boundary points as artificial
vertices, representing for instance a change in boundary conditions. Then a conical point $x$ is a smooth boundary point
if, and only if, $\omega_x\cong S_+^{n-1}$.
The minimal set of conical points is unique and
coincides with the singularities of $\pa \Omega$; these are \emph{true conical points} of $\Omega$. Here we will give our results for true vertices, but the constructions can easily be extended to artificial ones.

  For the remainder of the paper, we keep the notation as in Definition \ref{domain}. Moreover,
for a conical domain $\Omega$, we always denote by $$\Omega^{(0)}=\{ p_1, p_2, \cdots, p_l\},$$
the set of true conical points of $\Omega$, and by $\Omega_0$ be the smooth part of $\pa\Omega$,
i.e., $\Omega_0=\pa\Omega \backslash \{p_1,p_2,\cdots, p_l\}$.
We remark that we allow the bases $\omega_{p_{i}}$ and $\pa \omega_{p_{i}}$ to be disconnected (in fact, if $n=2$, $\pa \omega_{p_{i}}$ is always disconnected).

\vspace{0.2cm}

We now recall the definition  of the \emph{desingularization} $\Sigma(\Omega)$ of $\Omega$ of a conical domain without cracks, which is obtained from $\Omega$ by removing a,
possibly non-connected, neighborhood of the singular points and
replacing each connected component by a cylinder.
 We refer to  \cite{BMNZ} for details on this construction, see also \cite{CQ13, Kon, MelroseAPS}.

 \smallskip

We have  the following
\begin{equation*}\label{desing}
   \Sigma(\Omega) \cong \left(\bigsqcup\limits_{p_i \in
     \Omega^{(0)}}[0,1)\times \overline{\omega_{p_i}}
     \right) \;  \bigcup\limits_{\phi_{p_i} }\;
     \Omega,
\end{equation*}
where the two sets are glued by $\phi_i$ along a suitable neighborhood of $p_{i}$.

In the terminology of \cite{BMNZ}, the hyperfaces which are
 {\em not at infinity} correspond to actual faces of $\Omega$, denoted by $\pa'\Sigma(\Omega)$, that is,
\begin{equation}\label{bdry_notinfty}
 \pa'\Sigma(\Omega)  \cong \left(\bigsqcup\limits_{p_i \in
    \Omega^{(0)}}[0,1)\times \partial\omega_{p_i}  \right)
    \bigcup\limits_{\phi_{p_i}, \, p_i\in\Omega^{(0)}} \Omega_0.
\end{equation}
A hyperface {\em at infinity} corresponds to a singularity of $\Omega$.
Let $\pa^{''}\Sigma(\Omega)$ denote the union of hyperfaces at infinity.
Hence
\begin{equation}\label{bdry_infty}
  \pa'' \Sigma(\Omega) \cong \bigsqcup\limits_{p_i \in
    \Omega^{(0)}}
    \{0\}\times \overline{\omega_{p_i} }.
\end{equation}

The boundary $\pa\Sigma(\Omega)$ can be identified with the union of
$\pa'\Sigma(\Omega)$ and $\pa^{''}\Sigma(\Omega)$. Let $\Omega_0$ denote the smooth part of $\partial\Omega$, that is,
$\Omega_0:=\partial \Omega \backslash \Omega^{(0)}$.
Hence, we can write
\begin{eqnarray}\label{bdry.desing}
  \pa \Sigma(\Omega) &=&  \pa' \Sigma(\Omega) \cup  \pa'' \Sigma(\Omega) \notag \\
  &\cong& \left(\bigsqcup\limits_{p_i \in
    \Omega^{(0)}}[0,1)\times \partial\omega_{p_i} \cup
    \{0\}\times \overline{\omega_{p_i} } \right)
    \bigcup\limits_{\phi_{p_i}, \, p_i\in\Omega^{(0)}} \Omega_0.
\end{eqnarray}

We denote by $M:= \pa'\Sigma(\Omega)$.
Note that $M$ coincides with
the closure of $\Omega_0$ in $\Sigma(\Omega)$. It is a compact manifold with (smooth) boundary
 $$\pa M= \bigsqcup\limits_{p_i \in
    \Omega^{(0)}}\{0\} \times \partial\omega_{p_i}.$$
In fact, we regard $M:=\partial' \Sigma
(\Omega)$ as a desingularization of the boundary $\pa \Omega$. Operators on $M$ will be related to (weighted) operators on $\pa \Omega$, as we shall see in Subsection \ref{s.FredCond}.
See  \cite{BMNZ, CQ13} for more details.

\smallskip
\subsection{Groupoid construction for conical domains without cracks}\label{ss.groupoidnocrack}

Let $\Omega$ be a conical domain without cracks, $\Omega^{(0)}=\{ p_1, p_2, \cdots, p_l\}$
be the set of (true) conical points of $\Omega$, and $\Omega_0$ be the smooth part of $\pa\Omega$. We will review the definition of the \emph{layer potentials groupoid} $\maG \tto M$, with $M:=\pa'\Sigma(\Omega) $  a compact set, as in the previous subsection, following \cite{CQ13}.

\smallskip

Let $\maH:=[0,\infty) \rtimes (0,\infty) $ be the transformation groupoid
with the action of $(0,\infty)$ on $[0,\infty)$ by dilation (see Example \ref{ex:action_groupoid}).
To each $p_i \in \Omega^{(0)}$,
we first associate a groupoid $\maH \times (\pa\omega_{p_i})^2 \tto [0,\infty) \times \pa\omega_{p_i}  $,  where  $(\pa\omega_{p_i})^2$ is  the pair groupoid of $\pa\omega_{p_i}$ (see Example \ref{ex:pair_groupoid}).
 We then take its reduction to $[0,1) \times \pa\omega_{p_i}$ to define
 \begin{equation*}
   \maJ_{i}:= \left.\left( \maH \times (\pa\omega_{p_i})^2\right)\right|_{[0,1) \times \pa\omega_{p_i}}\tto [0,1) \times \pa\omega_{p_i}.
 \end{equation*}
 We now want to glue the pair groupoid $\Omega_0\times\Omega_0 = \Omega_0^2$ and the family $(\maJ_i)_{i=1,2,\ldots, l}$ in a suitable way. First, let $V_i \subset \mathbb{R}^n$ be a neighborhood of $p_i$ such that there is a diffeomorphism $\varphi_i:  (0,1)\times \pa\omega_{i} \cong V_{i}\cap\Omega_{0}$. Let $\varphi=(\varphi_{i})_{p_i\in \Omega^{(0)}}$ on the disjoint union $\bigsqcup_{i=1}^l V_i$, and set
 \begin{equation*}
   M = \left(\bigsqcup\limits_{p_i \in
    \Omega^{(0)}}[0,1)\times \partial\omega_{p_i} \right)
    \bigcup\limits_{\varphi} \Omega_0 = \d'\Sigma(\Omega),
 \end{equation*}
 as above.

 Note that $\maJ_i|_{(0,1)\times\pa\omega_{p_i}}$ is the pair groupoid $((0,1)\times\pa\omega_{p_i})^2$, so that
 \begin{equation*}
   \maJ_i|_{(0,1)\times\pa\omega_{p_i}} \cong \Omega_0^2|_{V_i}.
 \end{equation*}
 Moreover, it is easy to see that the family $(\maJ_i)_{i = 1}^l \cup \{\Omega_0^2\}$ satisfies the strong gluing condition of Subsection \ref{sub:gluing_groupoids} (the orbit of any point in $M$ is either $\Omega_0$ or one of the $\pa\omega_{p_i}$, for $i = 1,\ldots,l$). Therefore the gluing in the following definition is a well defined Hausdorff Lie groupoid.

\begin{definition}\label{gpd1}
Let $\Omega$ be a conical domain without cracks. The \emph{layer potentials groupoid} associated to $\Omega$ is the Lie groupoid $\maG \tto M:=\pa'\Sigma(\Omega)$ defined
by
\begin{equation}\label{grpd.nocrack}
  \cG:=\left(\bigsqcup\limits_{p_i \in \Omega^{(0)}}\maJ_{p_i} \right)\quad
  \bigcup\limits_{\varphi} \quad
 \Omega_0^2 \quad  \tto \quad M
\end{equation}
where $\varphi=(\varphi_{p_i})_{p_i\in \Omega^{(0)}}$, with space of units
\begin{eqnarray}\label{units.nocrack}
  M &=& \left(\bigsqcup\limits_{p_i \in
    \Omega^{(0)}}[0,1)\times \partial\omega_{p_i} \right)\quad
    \bigcup\limits_{\varphi} \quad \Omega_0 \quad  \cong \quad\partial' \Sigma (\Omega),
\end{eqnarray}
where $\partial' \Sigma (\Omega)$ (defined in Equation (\ref{bdry_notinfty})) denotes the union of hyperfaces
which are not at infinity of a desingularization.
\end{definition}

Clearly, the space $M$ of units is
compact. We have that
$\Omega_0$ coincides with  the interior of $M$, so $\Omega_0$ is an open dense subset of $M$.
The following proposition summarizes the properties of the layer potentials groupoid and its groupoid $C^{*}$-algebra. Note that  $C^*(\cH)=\maC_0([0,\infty))\rtimes \RR^+$, where $\RR^+=(0, \infty)$ is the multiplication group,
by \cite{MRen}.

\begin{proposition}\label{Groupoid}
Let $\cG$ be the layer potentials groupoid \eqref{grpd.nocrack} associated to a domain
with conical points $\Omega\subset \RR^n$. Let
$\Omega^{(0)}=\{p_1,p_2,\cdots, p_l\}$ be the set of conical points
and $\Omega_0=\partial \Omega \backslash \Omega^{(0)}$ be the smooth
part of $\partial\Omega$. Then, $\cG$ is a Lie groupoid with units
$M= \partial' \Sigma (\Omega)$ (defined in Equation (\ref{bdry_notinfty})) such that
\begin{enumerate}

\item $ \Omega_0$ is an open, dense invariant subset with
$\cG_{\Omega_0} \cong \Omega_0 \times \Omega_0$ and $\Psi^m(\cG_{\Omega_0})\cong \Psi^m(\Omega_0)$.

\item For each conical point $p \in \Omega^{(0)}$, the subset $\{ p \} \times \pa
\omega_p$ is $\G$-invariant and
\begin{equation*}
\cG_{\pa M}= \bigsqcup\limits_{i=1}^{l} (\pa\omega_i \times \pa\omega_i) \times \RR^+ \times \{p_i\}
\end{equation*}
  \item If $P\in \Psi^m(\cG_{\pa M})$ then for each $p_i\in \Omega^{(0)}$, $P$ defines a Mellin convolution operator on $\RR^+\times \pa \omega_i$.
\smallskip

\item $\cG$ is (metrically) amenable, i.e.\ $C^*(\cG)\cong C^*_r(\cG)$.

\item If $n\geq 3$, $C^*(\cG_{\pa M})\cong  \bigoplus\limits_{i=1}^{l} \maC_0(\RR^+) \otimes \maK$. If $n=2$, $C^*(\cG_{\pa M})\cong \bigoplus\limits_{i=1}^{l} M_{k_i}(\maC_0(\mathbb{R}^+)) $,
where $k_i$ is the number of elements of $\pa \omega_i$,
  the integer $l$ is the number of conical points, and $\maK$ is the algebra of compact operators on an infinite dimensional separable Hilbert space.
\end{enumerate}

\end{proposition}
Note that if $P\in \Psi^{m}(\maG)$ then, at the boundary, the regular representation yields an operator
$$P_{i}:= \pi_{p_i}(P) \in \Psi^m(\RR^+\times (\pa \omega_{i})^2),$$
where $\RR^+\times (\pa \omega_{i})^2$ is regarded as a groupoid (see Item (2) as above), which is defined by a distribution kernel $\kappa_i$ in $\RR^+\times (\pa
\omega_{i})^2$, hence a Mellin
convolution operator on $\RR^+\times \pa \omega_i$ with kernel
$\tilde{\kappa_{i}}(r,s, x',y'):=\kappa_i(r/s, x', y')$.
 If  $P\in \Psi^{-\infty}(\maG)$, that is, if $\kappa_i$ is smooth, then it defines a smoothing
{Mellin convolution operator} on $\RR^+\times \pa \omega_{i}$ (see
\cite{LP,QN12}). This is  one of the motivations in our
definition of $\cG$.

\begin{remark}\label{rmk.bgrpd}
Recall the definition of $b$-groupoid in Example~\ref{bgrpd}, which,
in the case of $M=\bigsqcup_{i} [0,1) \times \pa \omega_{i}$ comes down to
$$
  {^b\cG}=  \bigsqcup_{i,j} \;\RR^+\times (\pa_j \omega_{i})^2 \;\; \bigcup \;\; \Omega_0^2 \;
$$
where $\pa_j \omega_{i}$ denote the connected components of $\pa \omega_{i}$.
If $\pa \omega_{i}$ is connected,  for all $i=1,...,l$, then $\cG={^b\cG}$. In many cases of interest, $\pa \omega$ is not
connected, for instance, if $n=2$, that is,
if we have a polygonal domain, then $\pa \omega$ is  {always}
disconnected.
 In general, the groupoid
$\cG$ is larger and not $d$-connected, and ${^b\cG}$ is an open, wide
subgroupoid of $\cG$. (The main difference is that here we allow the different
connected components of the boundary, corresponding to the same conical point, to interact, in that there are arrows between
them.) The Lie algebroids of these two groupoids coincide, as $A(\cG)\cong {^bTM}$, the $b$-tangent bundle of $M$.
 Moreover,  $\Psi(\cG)\supset \Psi({^b\cG})$, and the later is the (compactly supported) $b$-pseudodifferential operators on $M$.
\end{remark}

\begin{remark}
  The construction of the layer potential groupoid can be extended to polygonal domains with cracks, when $n=2$, that is domains $\Omega \subset \R^{2}$ such that $\pa \Omega \neq \pa \overline{\Omega}$. This construction was done in \cite{CQ13}.

  The point is that in two dimensions, the actual cracks, given by $\pa \Omega \setminus \pa \overline{\Omega}$, are a collection of smooth crack lines and boundary points that behave like conical singularities (in higher dimensions we get \enquote{edges}).
  To each polygonal domain with cracks we can associate a generalized conical domain with \emph{no cracks}, the so-called unfolded domain,
  $$\Omega^u=\Omega \cup \pa^u\Omega,$$
   where $\pa^u \Omega$ is
the set of inward pointing unit normal vectors to  the smooth part of $\pa\Omega$.
The main idea is that a smooth crack point  should be covered by two points,
which correspond to the two sides of the crack. At the boundary we get a  double cover of each smooth crack line, and a $k$-cover of each singular crack point, $k$ being the ramification number (see \cite{CQ13} for details). The vertices of $\Omega$ are still vertices of the generalized domain, but now each singular crack point yields $k$ new vertices.
  The groupoid construction defined above still applies to this case, as all the results in this subsection and the next.
 
  \end{remark}

\subsection{Desingularization and weighted Sobolev spaces for conical domains}

An important class of function spaces on singular manifolds are weighted Sobolev spaces.
Let $\Omega$ be a conical domain, and $r_\Omega$ be the smoothened distant function to the set of conical points  $\Omega^{(0)}$
as in \cite{BMNZ, CQ13}.
The space $L^2(\Sigma(\Omega))$ is defined using the volume element of a compatible metric
on $\Sigma(\Omega)$. A natural choice of compatible metrics is $g= r^{-2}_\Omega \, g_e$, where
$g_e$ is the Euclidean metric. Then the Sobolev spaces $H^m(\Sigma(\Omega))$ are defined  in the usual way.
These Sobolev spaces can be identified with weighted Sobolev spaces.

\smallskip
Let $m\in \ZZ_{\geqslant 0}$ and $a\in\R$. The $m$-th
Sobolev space on $\Omega$ with weight $r_{\Omega}$ and index $a$ is
defined by
\begin{equation}\label{def.weighted}
  \maK_{a}^m(\Omega)=\{u\in L^2_{\text{loc}}(\Omega) \, | \,\,
  r_{\Omega}^{|\alpha|-a}\partial^\alpha u\in L^2(\Omega), \,\,\,\text{for
    all}\,\,\, |\alpha|\leq m\}.
\end{equation}
We defined similarly the spaces $ \maK_{a}^m(\pa\Omega)$. Note that in this case, as $\pa\Omega$ has no boundary, these spaces are defined for any $m\in \mathbb{R}$ {by complex interpolation \cite{BMNZ}}.

The following result is taken from \cite[Proposition~5.7 and
Definition 5.8]{BMNZ}.

\begin{proposition} \label{Identification}
Let $\Omega \subset \RR^n$ be a domain with conical points,
$\Sigma(\Omega)$ be its desingularization, and $\partial'\Sigma(\Omega)$
be the union of the hyperfaces that are not at infinity.
We have
\begin{enumerate}
\item[(a)] $ \maK^{m}_{\frac{n}{2}}(\Omega)\simeq H^{m}(\Sigma(\Omega),
  g),$ for all $m\in \mathbb{Z}$;
\item[(b)] $ \maK^{m}_{\frac{n-1}{2}}(\partial\Omega)\simeq
  H^{m}(\partial'\Sigma(\Omega),g)$, for all $m\in \mathbb{R}$.
\end{enumerate}
where the metric $g= r^{-2}_\Omega \, g_e$ with $g_e$ the Euclidean metric.
\end{proposition}

\subsection{Fredholm Conditions for Layer Potentials }\label{s.FredCond}

In this section, we relate the  layer potential groupoids for conical domains
constructed in Section \ref{s.LP_groupoids} with  boundary action groupoids. Moreover, we also show that they fit in the framework of Fredholm groupoids so we can apply the Fredholm criteria obtained in the previous sections operators on layer potential groupoids.

Recall the definition of boundary action groupoids from Subsection \ref{sub:boundary_action_groupoids}.
\begin{theorem}
The layer potentials groupoid defined in Definition \ref{gpd1} is a boundary action groupoid.
\end{theorem}
\begin{proof}
  The layer potential groupoids is build in several steps. First, the groupoid $\maH = [0,\infty) \rtimes (0,\infty)$ is obviously a boundary action groupoid. If $p_i$ is a conical point of $\Omega$, then $\d\omega_{p_i} \times\d\omega_{p_i}$ is also a boundary action groupoid (see Example \ref{ex:bag_pair_groupoid}). Hence $\maH \times (\d \omega_{p_i})^2$ is a boundary action groupoid by Theorem \ref{lm:bag_product}, and so is its reduction $\maJ_i = \left(\maH \times (\d \omega_{p_i})^2\right)|_{[0,1)\times\d \omega_{p_i}}$, according to Lemma \ref{lm:bag_reduction}. The layer potential groupoids is then obtained by gluing boundary action groupoids $\maJ_i$, for $i = 1,\ldots,l$, with the pair groupoid $\Omega_0 \times\Omega_0$, therefore it is again a boundary action groupoid by Lemma \ref{lm:bag_attaching_ends}.
\end{proof}

Let us now show that the layer potentials groupoid is Fredholm. The results of Subsection \ref{sub:fredholm_groupoids} do not apply here, so we use a more direct method.
Let us first see the case of straight cones. Let $\omega \subset S^{n-1}$ be an open subset with smooth
boundary (note that we allow $\omega$ to be \emph{disconnected}) and
$$\Omega := \{t y',\ y' \in \omega,\ t \in (0, \infty)\} =  \mathbb{R}^{+} \,\omega$$
be the (open, unbounded) cone with base $\omega$. The desingularization becomes  in this case  an
half-infinite solid cylinder
\begin{equation*}
\Sigma(\Omega)=  [0,\infty) \times \overline{\omega}
\end{equation*}
with boundary $\pa \Sigma(\Omega) = [0,\infty) \times \partial \omega
\cup \{0\}\times \omega$, so that  $M= \partial'
\Sigma(\Omega)=[0,\infty)\times \partial \omega$ the union of the
hyperfaces not at infinity. Taking the one-point compactification $[0,\infty]$ of $[0,\infty)$, we can consider the groupoid $\overline{\maH}$ as in Example \ref{expl.transfgroupoidFredholm}. Then the {layer potentials groupoid associated to a straight cone
  $\Omega\cong \mathbb{R}^+\omega$} is the product Lie groupoid with units
$M=[0,\infty] \times \pa \omega$, corresponding to a desingularization
  of $\pa \Omega$, defined as
\begin{equation*}\label{cJ}
\cJ := \overline{\maH} \times (\partial\omega)^2.
\end{equation*}
Now, we have seen in Example \ref{expl.transfgroupoidFredholm} that $\overline{\maH}$ is a Fredholm groupoid, hence $\maJ$ is also a Fredholm groupoid.

In the general case, we can proceed in several ways: we can use the same argument as in the straight cone case (that is, as in Example \ref{expl.transfgroupoidFredholm} ),  or we can use the fact that the gluing (along the interior) of Fredholm groupoids is also a Fredholm groupoid. By analogy with the classes of Fredholm groupoids studied in \cite{CNQ17}, we chose to check that $\maG$ is actually given by a fibered pair groupoid over the boundary.

\begin{theorem}\label{thm.LPFred}
The layer potentials groupoid defined in Definition \ref{gpd1} is a Fredholm groupoid.
\end{theorem}

\begin{proof}
Let us deal with the case of conical domains without cracks. The other case is similar, taking the unfolded boundary.

It is clear that $\Omega_0$ is an open, dense, $\maG$-invariant subset of $M= \pa'\Sigma(\Omega)$, with $\maG_{\Omega_{0}}$ is the pair groupoid. Let
$$F:= M\backslash \Omega_0 =\pa M = \bigcup\limits_{p \in \Omega^{(0)}} \{ p \} \times \pa \omega_{p}\cong  \bigsqcup\limits_{i=1}^{l} \pa \omega_{p_{i}}.$$
We have
\begin{equation*}
\cG_{F}= \bigsqcup\limits_{i=1}^{l} (\pa\omega_i \times \pa\omega_i) \times (\RR^+ )
\end{equation*}
For any $x\in F$, we have $(\maG_F)_x^x =\maG_x^x \simeq \{ x\} \times \mathbb{R}^+ \simeq \mathbb{R}^+ $.
Since the group $\mathbb{R}^+ $ is commutative, it is amenable.
We claim that $\maR(\maG_{F})=\{\pi_x,\, x \in F\}$ is a strictly spectral / exhaustive set of
  representations of $C\sp{\ast}(\maG_F)$.
This can be proved directly, using the description in (4) of Proposition \ref{Groupoid}.

We show alternatively that $\maG_{F}$ can be given as a fibered pair groupoid, along the lines of Example \ref{ex.help-for-lp}.
Let $\maP:= \left\{  \pa\omega_i  \right\}_{i=1,..., l}$ be a finite partition of the smooth manifold $F$ and let $f: F\to \maP$, $x\in  \pa\omega_i  \mapsto  \pa\omega_i  $.
Then each $\pa\omega_i$ is a closed submanifold of $F$ and $\maP$ is a smooth discrete manifold, with $f$ is a locally constant smooth fibration.

Let $\maH:= \maP\times \RR^{+}$, as a product of a manifold and a Lie group. Then, by Example \ref{ex.help-for-lp},
$$f\pullback(\maH) = \bigsqcup\limits_{i=1}^{l} (\pa\omega_i \times \pa\omega_i) \times \RR^+  =  \maG_{F}.$$
Hence, by Corollary \ref{prop.fred}, the result is proved.

\end{proof}

If we apply Theorem \ref{thm.nonclassical2} \cite[Theorem~4.17]{CNQ} to our case, we obtain the main theorems as follows. Recall that the regular representations $\pi_{x}$ and $\pi_{y}$ are unitarily equivalent for $x,y$ in the same orbit of $\maG_{F}$, so that for $P\in \Psi^{m}(\maG)$  we obtain a family of Mellin convolution operators
$P_{i}:=\pi_{x}(P)$ on $\RR^+\times \pa\omega_i$, $i=1,...,p$, with
$x=(p_{i},x')   \in \pa M$, $x'\in \pa\omega_{p_{i}}$.

Recall that the space $L^{m}_{s}(\maG)$
is the {\em norm closure} of $\Psi^m(\maG)$ in the topology of
continuous operators $H^s(M )\to H^{s-m}(M)$. By the results in \cite{QL18, QN12}, if $P\in  L^{m}_{s}(\maG)$, then $\pi_{p_{i}}(P)$ is also a Mellin convolution operator.

\begin{theorem}\label{thm.fredholm}
Suppose that $\Omega \subset \mathbb{R}^n$ is a conical domain without cracks and
$\Omega^{(0)}=\{p_1,p_2,\cdots, p_l\}$ is the set of conical points. Let $\maG\tto M=\pa'\Sigma(\Omega)$
be the layer potentials groupoid as in Definition \ref{gpd1}.
 Let $P\in L^{m}_{s}(\maG)\supset \Psi^m(\maG)$ and $s \in \RR $.
Then
\begin{equation*}
  P : \maK_{\frac{n-1}{2}}^s(\pa\Omega) \to \maK_{\frac{n-1}{2}}^{s-m}(\pa\Omega)
\end{equation*}
is Fredholm if, and only if,
\begin{enumerate}
  \item $P$ is elliptic and
  \item the Mellin convolution operators
    \begin{equation*}
      P_{i} : H^s(\RR^+\times \pa\omega_i;g) \to H^{s-m}(\RR^+\times \pa\omega_i;g)
    \end{equation*}
    are invertible, for $i=1,\ldots,p$, where the metric $g= r^{-2}_\Omega \, g_e$ with $g_e$ the Euclidean metric.
\end{enumerate}
\end{theorem}

\begin{rem}
  Fredholm conditions similar to those of Theorem \ref{thm.fredholm} also hold for polygonal domains with cracks. In that case, some extra limit operators arise from the fact that the boundary $\d \Omega$ should be desingularized near the crack points.
\end{rem}

We expect that these results have applications to layer potentials.

\vspace{0.3cm}

\printbibliography

\end{document}